\documentclass[a4paper]{amsart}
\usepackage[centertags]{amsmath}
\usepackage{hyperref}
\usepackage[textwidth=15cm,hcentering]{geometry}
\usepackage[utf8]{inputenc}
\usepackage[english]{babel}
\usepackage{babel}
\usepackage{units}
\usepackage{amstext}
\usepackage{amsthm}
\usepackage{amssymb}

\usepackage{biblatex}
\makeatletter
\numberwithin{equation}{section}
\numberwithin{figure}{section}
\theoremstyle{plain}
\newtheorem{thm}{\protect\theoremname}[section]
  \theoremstyle{remark}
  \newtheorem{rem}[thm]{\protect\remarkname}
  \theoremstyle{plain}
  \newtheorem*{thm*}{\protect\theoremname}
  \theoremstyle{plain}
  \newtheorem{prop}[thm]{\protect\propositionname}
  \theoremstyle{definition}
  \newtheorem{defn}[thm]{\protect\definitionname}
  \theoremstyle{plain}
  \newtheorem{lem}[thm]{\protect\lemmaname}
  \theoremstyle{plain}
  \newtheorem{cor}[thm]{\protect\corollaryname}

\@ifundefined{date}{}{\date{}}
\usepackage{tensind}
\tensordelimiter{?}
\usepackage{braket}

\makeatother

  \providecommand{\corollaryname}{Corollary}
  \providecommand{\definitionname}{Definition}
  \providecommand{\lemmaname}{Lemma}
  \providecommand{\propositionname}{Proposition}
  \providecommand{\remarkname}{Remark}
  \providecommand{\theoremname}{Theorem}
\providecommand{\theoremname}{Theorem}

\begin{document}
\global\long\def\real{\mathbb{R}}
\global\long\def\nat{\mathbb{N}}
\global\long\def\q{\mathbb{Q}}
\global\long\def\z{\mathbb{Z}}
\global\long\def\c{\mathbb{C}}

\global\long\def\d{\mathrm{d}}
\global\long\def\p{\mathbb{P}}

\global\long\def\norm#1{\left\Vert #1\right\Vert }

\global\long\def\g{\mathfrak{g}}

\global\long\def\h{\mathbb{H}}
\global\long\def\id{\mathrm{id}}

\global\long\def\im{\operatorname{Im}}

\global\long\def\hom{\operatorname{Hom}}

\global\long\def\rep{\operatorname{Rep}}

\global\long\def\endd{\operatorname{End}}

\global\long\def\sl{\mathrm{SL}}
\global\long\def\psl{\mathrm{PSL}}

\global\long\def\gl{\mathrm{GL}}
 \global\long\def\so{\mathrm{SO}}

\global\long\def\o{\mathcal{O}}

\global\long\def\supp{\operatorname{supp}}
\global\long\def\isom{\operatorname{Isom}}

\global\long\def\sgn{\operatorname{sign}}

\global\long\def\sla{\mathfrak{sl}}

\global\long\def\vol{\operatorname{Vol}}

\global\long\def\tr{\operatorname{tr}}

\title{Homogeneous Asymptotic Limits of Uniform Averages on Fuchsian Groups}

\author{Tamir Hemo}
\address{Department of Mathematics, Technion - Israel Institute of Technology, Haifa, Israel 32000}
\begin{abstract}
We show that averages on geometrically finite Fuchsian groups, when
embedded via a representation into a space of matrices, have a homogeneous
asymptotic limit under appropriate scaling. This generalizes some
of the results of Maucourant to subgroups of infinite co-volume. The resulting disrtibution is expressed in terms of a measure considered by Mohammadi and Oh. 
\end{abstract}

\maketitle

\section{Introduction}

Let $G$ be a topological group and $\Gamma<G$ a discrete subgroup.
The asymptotic properties of matrix elements of $\Gamma$
evaluated in various representations of $G$ are related to many different problems
in mathematics, from expander graphs to Diophantine approximation, the Zaremba conjecture and Apollonian circles
See \cite{kontorovich2013apollonius},  \cite{bourgain2014zaremba}, and \cite{fuchs2014ubiquity}, for more. Such problems can be effectively studied using harmonic analysis of the corresponding symmetric spaces \cite{oh2014harmonic}.

One method of studying the asymptotic properties of matrix elements
in linear groups was initiated by Maucourant \cite{maucourant2007homogeneous},
who considered the following question. Let $G$ be a non-compact connected
semisimple Lie group with finite center and let $\rho:G\rightarrow\gl\left(V\right)$
be a faithful representation of $G$ in a finite dimensional real
vector space $V$. Let $\d g$ denote a Haar measure on $G$. Let
$f$ be a compactly supported continuous function on $\endd\left(V\right)$.
Maucourant showed that for some rational number $d$, and a positive
integer $e$ between $0$ and $\operatorname{rank}_{\real}G-1$, the
measures
\[
\mu_{G,T}\left(f\right)=\frac{1}{T^{d}\log\left(T\right)^{e}}\int_{G}f\left(\frac{\rho\left(g\right)}{T}\right)\,\d g,\quad f\in C_{c}\left(\endd\left(V\right)\right).
\]
on $\endd\left(V\right)$ converge weakly as $T$ tends to infinity.
Moreover, $e$ and $d$ can be described explicitly, as is the limit
measure $\mu_{\infty}$. Moreover, $\mu_{\infty}$ is homogeneous
of degree $d$ in the sense that for any Borel set $E\subseteq\endd\left(V\right)$
we have $\mu_{\infty}\left(tE\right)=t^{d}\mu_{\infty}\left(E\right)$.

Let $\Gamma$ be an irreducible lattice in $G$, and for $T>0$ consider
the measure:
\begin{align*}
\mu_{\Gamma,T} & =\frac{1}{T^{d}\log\left(T\right)^{e}}\sum_{\gamma\in\Gamma}\delta_{\frac{\rho\left(\gamma\right)}{T}}.
\end{align*}
It is possible to compare $\mu_{\Gamma,T}$ with $\mu_{G,T}$ (see
\cite[section 3]{maucourant2007homogeneous}) and get that as $T\rightarrow\infty$
the measures $\mu_{\Gamma,T}$ converge weakly to
\[
\lim_{T\rightarrow\infty}\mu_{\Gamma,T}=\frac{1}{\vol\left(\Gamma\backslash G\right)}\mu_{\infty},
\]
where $\vol\left(\Gamma\backslash G\right)$ is the volume of a fundamental
domain of $\Gamma$ with respect to the chosen Haar measure on $G$.

In this work we wish to study averaging operators similar to $\mu_{\Gamma,T}$
for $G=\sl_{2}\left(\real\right)$ and discrete subgroups $\Gamma<\sl_{2}\left(\real\right)$
which are not necessarily of finite co-volume. In this case it is
not possible to compare the measures $\mu_{\Gamma,T}$ and $\mu_{\sl_{2}\left(\real\right),T}$
for a general discrete subgroup $\Gamma$ as the asymptotics in the number of points in $\Gamma$ will be different. The normalization of the average should be proportional to the asymptotics of $\left|\rho\left(\Gamma\right)\cap B_{T}\right|$, with $B_{\mathrm{T}}$ being a ball of raduis $T$ with respect to a bi-$K$ invariant norm.
By a result of Lax and Phillips \cite[theorem 1]{lax1982asymptotic}, if $\Gamma$ is geometrically finite
\[
\left|\rho\left(\Gamma\right)\cap B_{T}\right|\sim T^{\delta d}
\]
as $T\rightarrow \infty$ where $\delta$ is the critical exponent of $\Gamma$.

We can then consider the measures
\begin{equation}
\mu_{\Gamma,T}=\frac{1}{T^{\delta d}}\sum_{\gamma\in\Gamma}\delta_{\frac{\rho\left(\gamma\right)}{T}},\quad T>0.\label{eq:intro}
\end{equation}
In this paper we prove that if $\delta>\frac{1}{2}$, then as $T\rightarrow \infty$ the measures
$\mu_{\Gamma,T}$ converge weakly to a nonzero measure
$\mu_{\Gamma}$ on $\endd\left(V\right)$ first considered by Mohammadi and Oh \cite[Definition 1.8.]{Mohammadi_2015}.
The measure $\mu_{\Gamma}$ is homogeneous of degree $\delta d$ and can be described explicitly in terms of $\rho$ and the Patterson-Sullivan
measure on the limit set of $\Gamma$. 

A study of the angular distribution of orbits in $\real^{2}$ of finitely
generated subgroup of $\sl_{2}\left(\real\right)$ was carried by
Maucourant and Schapira \cite{mauschapira} using ergodic properties
of the horocycle flow. The present work gives an alternative proof
for some of their results corresponding to a similar normalization
as the one we are considering. 

Our argument follows the general paradigm described in \cite{Mohammadi_2015}, specialized to $\mathrm{SL}_{2}(\mathbb{R})$ together with the argument of \cite{maucourant2007homogeneous}. Namely, we decompose the sum \ref{eq:intro} into averages
over annuli depending on the norm of the elements in $\rho\left(\gamma\right)$.
Each individual average over an annulus can then be analyzed using estimates \cite[theorem 1.14]{Mohammadi_2015} on Fourier coefficients. In the case we consider, the relevant spectral estimates were first established by Bourgain, Kontorovich, and Sarnak \cite[theorem 1.5.]{bourgain2010sector}. 

\subsubsection*{Acknowledgements}

The author would like to thank Amos Nevo for his guidance and advice
during the course of this work and to Hee Oh for her comments on an earlier draft of this paper. 
\tableofcontents{}
\section{\label{sec:The-main-results}The Main Results}

Let $\Gamma<\sl_{2}\left(\real\right)$ be a discrete subgroup containing
$-1$ such that the image $\overline{\Gamma}$ of $\Gamma$ in $\psl_{2}\left(\real\right)$
is a non-elementary geometrically finite Fuchsian group of the second
kind with critical exponent $\delta>\frac{1}{2}$. Let $\mathbb{H}^{2}$
be the hyperbolic plane and let $\overline{\mathbb{H}^{2}}=\mathbb{H}^{2}\cup\partial\mathbb{H}^{2}$
be its compactification with $\partial\mathbb{H}^{2}$ the Gromov
boundary. In the Poincaré disk model $\partial\mathbb{H}^{2}$ can
be identified with the unit circle and in the upper half plane model
with $\real\cup\left\{ \infty\right\} $.

For the group $\Gamma$ and any two points $x,y\in\mathbb{H}^{2}$,
Patterson \cite{patterson1976limit} (see also \cite{sullivan1981ergodic})
defined a measure $\mu_{x,y}$ on the limit set $\Lambda\left(\Gamma\right)$
of $\Gamma$. This limit set consists of all points in $\partial\mathbb{H}^{2}$
that lie in the closure of one (equivalently, all) orbit of $\Gamma$
on $\mathbb{H}^{2}$. Since $\Gamma$ has exponent $\delta>\frac{1}{2}$,
the measure $\mu_{x,y}$ is obtained as the weak limit:
\[
\mu_{x,y}=\lim_{s\rightarrow\delta^{+}}\mu_{s,x,y},
\]
of measures $\mu_{s,x,y}$ on $\overline{\mathbb{H}^{2}}=\mathbb{H}^{2}\cup\partial\mathbb{H}^{2}$
given by
\[
\mu_{s,x,y}=\frac{1}{\sum_{\gamma\in\Gamma}e^{-s\d\left(x,\gamma y\right)}}\sum_{\gamma\in\Gamma}e^{-s\d\left(x,\gamma y\right)}\delta_{\gamma.y},
\]
where $\d$ is the hyperbolic metric on $\mathbb{H}^{2}$. Recall
that $\mathbb{H}^{2}$ carries a transitive action of $\sl_{2}\left(\real\right)$
realized via fractional linear transformations in the upper half plane
model. The stabilizer of $i$ is $K:=\so_{2}\left(\real\right)$.
We then identify $\mathbb{H}^{2}$ with $\nicefrac{\sl_{2}\left(\real\right)}{\so_{2}\left(\real\right)}$.
Let $o$ denote the point stabilized by $K$.

Define:
\begin{align*}
K & =\left\{ k_{\theta}=\begin{pmatrix}\cos\theta & \sin\theta\\
-\sin\theta & \cos\theta
\end{pmatrix}|0\le\theta<2\pi\right\} ,\\
N & =\left\{ \begin{pmatrix}1 & x\\
0 & 1
\end{pmatrix}|x\in\real\right\} \\
A^{+} & =\left\{ a_{t}=\begin{pmatrix}e^{\frac{t}{2}} & 0\\
0 & e^{-\frac{t}{2}}
\end{pmatrix}|t\ge0\right\} .
\end{align*}
Recall the Cartan decomposition $\sl_{2}\left(\real\right)=KA^{+}K$.
Given $g=k_{1}a_{t}k_{2}$ for $g\in\sl_{2}\left(\real\right)$, the
parameter $t$ is uniquely determined, and for $t\neq0$ the elements
$k_{1},k_{2}$ are determined up to multiplication by $\pm I$, where
$I\in\sl_{2}\left(\real\right)$ is the identity.

Let $\rho:\sl_{2}\left(\real\right)\rightarrow\gl\left(V\right)$
be a faithful representation on a finite dimensional real vector space
$V$ . Let
\[
\left(V,\rho\right)=\bigoplus_{i}\left(V_{k_{i}},\rho_{k_{i}}\right)
\]
be the decomposition of $V$ into a direct sum of irreducible representations,
where $\rho_{k_{i}}$ is the irreducible representation of highest
weight $k_{i}$. Let $k$ be the highest weight occurring in $V$,
so that $k=\max_{i}k_{i}$. We denote by $m$ the multiplicity of
$V_{k}$ in $V$. Let $\left\langle .,.\right\rangle :V\rightarrow\real$
be an inner product on $V$ which is $\rho\left(K\right)$-invariant
and such that there exists an orthonormal basis of $V$ with respect
to $\left\langle .,.\right\rangle $ in which $\rho\left(A\right)$
is diagonal. Such an inner product always exists and for any $t\ge0$
any such inner product satisfies (see \ref{sec:K-invariant-norms}
for details)
\[
\norm{\rho\left(a_{t}\right)}_{\mathrm{op}}=\sup_{0\neq v\in V}\frac{\norm{\rho\left(a_{t}\right)v}}{\norm v}=e^{\frac{kt}{2}}.
\]
We fix such an inner product and endow $\endd\left(V\right)$ with
the operator norm, denoted $\norm .:\endd\left(V\right)\rightarrow\real$.

Let $P_{k}:V\rightarrow V$ be the orthogonal projection onto the
subspace of vectors in $V$ of weight $k$. That is, $P_{k}|_{V_{k_{i}}}=0$
if $k_{i}\neq k$, and on each copy of $V_{k}$ in $V$, $P_{k}|_{V_{k}}$
is the projection onto the one dimensional space consisting of highest
weight vectors. For example, if $\rho$ is the standard representation
of $\sl_{2}\left(\real\right)$, which is an irreducible representation
of highest weight 1,
\[
P_{1}=\begin{pmatrix}1 & 0\\
0 & 0
\end{pmatrix}.
\]
Let $\delta>\frac{1}{2}$ and define
\[
V_{\Gamma}=2\pi^{\frac{1}{2}}\frac{\Gamma\left(\delta-\nicefrac{1}{2}\right)}{\Gamma\left(\delta+1\right)},
\]
where $\Gamma\left(s\right)$ for $s\in\c$ on the right hand side
is the value Gamma function at $s$. We are now ready to state our
main result.
\begin{thm}
\label{thm:unnorm}Let $\Gamma,\rho,k$, and $P_{k}$ be as above.
There exists a measure $\mu_{\Gamma}$ on $\endd\left(V\right)$ such
that for any compactly supported continuous function $f:\endd\left(V\right)\rightarrow\c$,
\[
\lim_{T\rightarrow\infty}\frac{1}{T^{\frac{2\delta}{k}}}\sum_{\gamma\in\Gamma}f\left(\frac{\rho\left(\gamma\right)}{T}\right)=\int_{\endd\left(V\right)}f\,\d\mu_{\Gamma}.
\]
That is, the measures
\[
\mu_{\Gamma,T}=\frac{1}{T^{\frac{2\delta}{k}}}\sum_{\gamma\in\Gamma}\delta_{\frac{\rho\left(\gamma\right)}{T}},
\]
converge to $\mu_{\Gamma}$ in the weak-{*} topology of $\left(C_{c}\left(\endd\left(V\right)\right)\right)^{*}$.
Moreover, $\mu_{\Gamma}$ is homogeneous of degree $\frac{2\delta}{k}$
and it is given by:
\[
\mu_{\Gamma}\left(f\right)=\frac{\delta}{2k}V_{\Gamma}\int_{K\times\left[0,\infty\right)\times K}f\left(\rho\left(k_{1}\right)\cdot tP_{k}\cdot\rho\left(k_{2}\right)\right)t^{\frac{2\delta}{k}-1}\,\d\mu\left(k_{1}\right)\d t\d\mu\left(k_{2}\right).
\]
where $\mu$ is a symmetric lift of the Patterson-Sullivan measure
$\mu_{o,o}$ from $\partial\mathbb{H}^{n}$ to $K$ (see remark \ref{rem:symmetric}).
\end{thm}
\begin{rem}
\label{rem:symmetric}We parametrize $K=\so\left(2,\real\right)$
by $\left[0,2\pi\right]$ via
\[
K=\left\{ k_{\theta}=\begin{pmatrix}\cos\theta & \sin\theta\\
-\sin\theta & \cos\theta
\end{pmatrix}|0\le\theta\le2\pi\right\} .
\]
This allows us to identify $\partial\mathbb{H}^{2}$ with $\left[0,\pi\right]$.
This identification is consistent with the Cartan coordinates on $\sl_{2}\left(\real\right)$,
since an elliptic transformation $k_{\theta}$ acts as a rotation
by $2\theta$ on hyperbolic space. The Patterson-Sullivan measure
is constructed for discrete subgroups of $\psl_{2}\left(\real\right)$.
For $g\in\psl_{2}\left(\real\right)$ we define the angles $\theta_{1}\left(g\right),\theta_{2}\left(g\right)\in\left[0,\pi\right]$
to be such that $g=\pm k_{\theta_{1}\left(g\right)}a_{t}k_{\theta_{2}\left(g\right)}$
for some $t\ge0$. Then the Patterson-Sullivan measure $\mu_{o,o}$
can be identified with a measure on $\left[0,\pi\right]$. We extend
$\mu_{o,o}$ to $\left[0,2\pi\right]$ by setting it to be symmetric
with respect to $\theta\mapsto\theta+\pi$. We denote the extended
measure by $\mu$. This is consistent with the notation of Bourgain,
Kontorovich, and Sarnak in \cite{bourgain2010sector}. In particular,
if $\Gamma$ is a lattice then in this convention the Patterson-Sullivan
measure on $K$ has total mass $2$ and is proportional to the Lebesgue
measure on $\left[0,2\pi\right]$.
\end{rem}
\begin{rem}
The limit measure of the theorem was originally considered by considered by Mohammadi and Oh \cite[Definition 1.8.]{Mohammadi_2015}.
\end{rem}
We can also state a normalized version of the main theorem, which
can be more useful in certain situations. It is easy to see that the
two versions are equivalent.
\begin{thm}
\label{thm:normalized}Let $\Gamma,\rho,P_{k}$ be as in \ref{thm:unnorm}.
Let $\norm .:V\rightarrow\real$ be the norm on $V$ fixed above.
Define:
\[
\Gamma_{T}=\left\{ \gamma\in\Gamma|\norm{\rho\left(\gamma\right)}\le T\right\} .
\]
There exists a measure $\nu_{\Gamma}$ on $\endd\left(V\right)$ such
that for any continuous $f:\endd\left(V\right)\rightarrow\c$,
\[
\lim_{T\rightarrow\infty}\frac{1}{T^{\frac{2\delta}{k}}}\sum_{\gamma\in\Gamma_{T}}f\left(\frac{\rho\left(\gamma\right)}{T}\right)=\int_{\endd\left(V\right)}f\,\d\nu_{\Gamma}.
\]
Moreover, $\nu_{\Gamma}$ is given by:
\[
\nu_{\Gamma}\left(f\right)=\frac{\delta}{2k}V_{\Gamma}\int_{K\times\left[0,1\right]\times K}f\left(\rho\left(k_{1}\right)\cdot tP_{k}\cdot\rho\left(k_{2}\right)\right)t^{\frac{2\delta}{k}-1}\,\d\mu\left(k_{1}\right)\d t\d\mu\left(k_{2}\right).
\]
\end{thm}
A quantitative estimate for the error term in the theorems above can
be given for H\"{o}lder continuous functions. Let
\[
0\le\lambda_{0}<\lambda_{1}<\cdots<\lambda_{N}<\frac{1}{4}
\]
be the eigenvalues of the Laplacian on $\Gamma\backslash\mathbb{H}$
below $\frac{1}{4}$ (see section \ref{sec:Sector-estimates}). Write
\[
\lambda_{j}=s_{j}\left(1-s_{j}\right)
\]
with $s_{j}>\frac{1}{2}$. 
\begin{thm}
\label{thm:Main_Quant}(same assumptions as theorems \ref{thm:unnorm}
and \ref{thm:normalized}) Let $f:\endd\left(V\right)\rightarrow\c$
be H\"{o}lder continuous with exponent $\alpha\in\left(0,1\right]$ and
constant $\norm f_{\mathrm{Lip}\alpha}$ .
\begin{enumerate}
\item Assume $f$ has compact support. Then, as $T$ tends to infinity,
\[
\frac{1}{T^{\frac{2\delta}{k}}}\sum_{\gamma\in\Gamma}f\left(\frac{\rho\left(\gamma\right)}{T}\right)=\int_{\endd\left(V\right)}f\,\d\mu_{\Gamma}+\o\left(\left(\norm f_{\infty}+\norm f_{\mathrm{Lip}\alpha}\right)\left(T^{\frac{\alpha}{2k}\left(s_{1}-\delta\right)}+T^{\frac{\alpha}{16k}\left(1-2\delta\right)}\right)\log\left(T\right)\right),
\]
where $\norm f_{\infty}=\sup_{x\in V}\left|f\left(x\right)\right|$,
and the implied constants depend only on $\Gamma$, $\rho$ and the
norm on $V$.
\item If $f$ is not necessarily of compact support, then, as $T\rightarrow\infty$
\[
\frac{1}{T^{\frac{2\delta}{k}}}\sum_{\gamma\in\Gamma_{T}}f\left(\frac{\rho\left(\gamma\right)}{T}\right)=\int_{\endd\left(V\right)}f\,\d\nu_{\Gamma}+\o\left(\left(\norm f_{\infty,1}+\norm f_{\mathrm{Lip}\alpha}\right)\left(T^{\frac{\alpha}{2k}\left(s_{1}-\delta\right)}+T^{\frac{\alpha}{16k}\left(1-2\delta\right)}\right)\log\left(T\right)\right),
\]
 where $\norm f_{\infty,1}:=\sup_{\norm x\le1}\left|f\left(x\right)\right|$. 
\end{enumerate}
\end{thm}
One application of the main theorem involves the study of the distribution
of individual matrix elements of the group $\Gamma$. For example,
let $\rho:\sl_{2}\left(\real\right)\rightarrow\gl_{2}\left(\real\right)$
be the standard representation and let $\norm .$ denote the operator
norm induced by the Euclidean norm on $\real^{2}$. This norm is given
by
\[
\norm{\begin{pmatrix}a & b\\
c & d
\end{pmatrix}}=\max\left\{ \left|\alpha\right|,\left|\beta\right|\right\} ,
\]
where
\[
\begin{pmatrix}a & b\\
c & d
\end{pmatrix}=k_{\theta}\begin{pmatrix}\alpha & 0\\
0 & \beta
\end{pmatrix}k_{\phi}
\]
for some $\theta,\phi\in\left[0,2\pi\right]$. Let $\Gamma_{T}$ denote
the set of matrices in $\Gamma$ with norm smaller than $T$, where
we identify $\sl_{2}\left(\real\right)$ with its image in $\gl_{2}\left(\real\right)$.
For $T>0$, consider the set of values 
\[
E_{T}=\left\{ \frac{a}{T}\,|\,\begin{pmatrix}a & b\\
c & d
\end{pmatrix}\in\Gamma_{T}\right\} .
\]
$E_{T}$ is a subset of the interval $\left[-1.1\right]$, and one
can ask whether this set of values is equidistributed in $\left[-1,1\right]$,
i.e., whether the limit
\[
\lim_{T\rightarrow\infty}\frac{1}{\left|\Gamma_{T}\right|}\sum_{r\in E_{T}}f\left(r\right)=\lim_{T\rightarrow\infty}\frac{1}{\left|\Gamma_{T}\right|}\sum_{\gamma\in\Gamma_{T}}f\left(\frac{a}{T}\right),
\]
exists for any continuous function $f:\left[-1,1\right]\rightarrow\c$. 
\begin{thm}
For any continuous function $f:\left[-1,1\right]\rightarrow\c$
\[
\lim_{T\rightarrow\infty}\frac{1}{\left|\Gamma_{T}\right|}\sum_{\gamma\in\Gamma_{T}}f\left(\frac{a}{T}\right)=\frac{\delta}{2}\int_{\left[0,2\pi\right]\times\left[0,1\right]\times\left[0,2\pi\right]}f\left(t\cos\theta_{1}\cos\theta_{2}\right)t^{2\delta-1}\,\d\mu\left(\theta_{1}\right)\d t\d\mu\left(\theta_{2}\right).
\]
\end{thm}
\begin{proof}
Applying theorem \ref{thm:normalized} to the function 
\[
\begin{pmatrix}a & b\\
c & d
\end{pmatrix}\mapsto f\left(a\right)
\]
on $\endd\left(\real^{2}\right)$, we see that for any continuous
$f$ as above,
\[
\lim_{T\rightarrow\infty}\frac{1}{T^{2\delta}}\sum_{\gamma\in\Gamma_{T}}f\left(\frac{a}{T}\right)=\frac{\delta}{2}V_{\Gamma}\int_{C}f\left(k_{\theta_{1}}\begin{pmatrix}t & 0\\
0 & 0
\end{pmatrix}k_{\theta_{2}}\right)t^{2\delta-1}\,\d\mu\left(\theta_{1}\right)\d t\d\mu\left(\theta_{2}\right),
\]
where $C=\left[0,2\pi\right]\times\left[0,1\right]\times\left[0,2\pi\right]$.
Since the matrix elements are given by
\begin{align*}
k_{\theta_{1}}\begin{pmatrix}t & 0\\
0 & 0
\end{pmatrix}k_{\theta_{2}} & =\begin{pmatrix}t\cos\theta_{1}\cos\theta_{2} & t\cos\theta_{1}\sin\theta_{2}\\
-t\sin\theta_{1}\cos\theta_{2} & -t\sin\theta_{1}\sin\theta_{2}
\end{pmatrix},
\end{align*}
and the size of $\left|\Gamma_{T}\right|$ by 
\[
\lim_{T\rightarrow\infty}\frac{\left|\Gamma_{T}\right|}{T^{2\delta}}=V_{\Gamma},
\]
We get,
\[
\lim_{T\rightarrow\infty}\frac{1}{\left|\Gamma_{T}\right|}\sum_{\gamma\in\Gamma_{T}}f\left(\frac{a}{T}\right)=\frac{\delta}{2}\int_{\left[0,2\pi\right]\times\left[0,1\right]\times\left[0,2\pi\right]}f\left(t\cos\theta_{1}\cos\theta_{2}\right)t^{2\delta-1}\,\d\mu\left(\theta_{1}\right)\d t\d\mu\left(\theta_{2}\right).
\]
\end{proof}
Hence, $\left\{ E_{T}\right\} $ is equidistributed in $\left[-1,1\right]$
with respect to the measure induced by the projection from the ``
limit cone''
\[
C\left(\Gamma,1\right)=\left\{ k_{\theta_{1}}\begin{pmatrix}t & 0\\
0 & 0
\end{pmatrix}k_{\theta_{2}}|\theta_{1},\theta_{2}\in\left[0,2\pi\right],0\le t\le1\right\} 
\]
onto the $a$-axis. If $\Gamma$ happens to be a lattice, the distribution
is given by:
\[
\lim_{T\rightarrow\infty}\frac{1}{\left|\Gamma_{T}\right|}\sum_{\gamma\in\Gamma_{T}}f\left(\frac{a}{T}\right)=2\cdot\int_{\left[0,2\pi\right]\times\left[0,1\right]\times\left[0,2\pi\right]}f\left(t\cos\theta_{1}\cos\theta_{2}\right)t\,\d\theta_{1}\d t\d\theta_{2}.
\]

\section{Preliminaries}

This section reviews some of the results needed in the proof of the
main theorem and sets up notation. We give a brief review of the sector
estimates results of Bourgain, Kontorovich and Sarnak from \cite{bourgain2010sector}
and recall some basic facts about $\sl_{2}\left(\real\right)$ and
its finite dimensional representations.

\subsection{\label{sec:Basic-structure-and-rep=00003DSL2}Finite dimensional
representations of $\protect\sl_{2}\left(\protect\real\right)$}

For proofs of the facts stated here, see chapter II of \cite{knapp2016representation}.

Recall that we have defined:
\begin{align*}
K & =\left\{ k_{\theta}=\begin{pmatrix}\cos\theta & \sin\theta\\
-\sin\theta & \cos\theta
\end{pmatrix}|0\le\theta<2\pi\right\} ,\\
N & =\left\{ \begin{pmatrix}1 & x\\
0 & 1
\end{pmatrix}|x\in\real\right\} \\
A^{+} & =\left\{ a_{t}=\begin{pmatrix}e^{\frac{t}{2}} & 0\\
0 & e^{-\frac{t}{2}}
\end{pmatrix}|t\ge0\right\} .
\end{align*}
We recall the Cartan decomposition for the particular case of $\sl_{2}\left(\real\right)$.
\begin{thm*}
\label{(Cartan-decomposition)}($KAK$ decomposition) The map $K\times A^{+}\times K\rightarrow\sl_{2}\left(\real\right)$
given by $\left(k_{\theta_{1}},a_{t},k_{\theta_{2}}\right)\mapsto k_{\theta_{1}}a_{t}k_{\theta_{2}}$
is surjective. If $t\neq0$, $k_{1}$ is uniquely determined up to
multiplication on the right by $\pm I$.
\end{thm*}
Even though the decomposition is not unique, we will call the coordinates
$\left(\theta_{1},t,\theta_{2}\right)$, with $\theta_{1},\theta_{2}\in\left[0,2\pi\right]$
and $t\ge0$, the \emph{Cartan coordinates }on $\sl_{2}\left(\real\right)$.
We denote by $g\left(\theta_{1},t,\theta_{2}\right)$ the element
$k_{\theta_{1}}a_{t}k_{\theta_{2}}$. Recall that in terms of these
coordinates a Haar measure on $\sl_{2}\left(\real\right)$ is given
by

\[
\mu\left(f\right)=\int_{K\times\left(0,\infty\right)\times K}f\left(k_{\theta_{1}}a_{t}k_{\theta_{2}}\right)\sinh\left(t\right)\,\d\theta_{1}\d t\d\theta_{2}.
\]

\subsubsection{Irreducible representations}

All finite dimensional irreducible representations of $\sl_{2}\left(\real\right)$
can be obtained as follows. Fix an integer $n\ge0$ and let $V_{n}$
be the complex vector space of polynomials in $\c\left[z_{1},z_{2}\right]$
which are homogeneous of degree $n$. Define a representation $\rho_{n}$
of $\sl_{2}\left(\real\right)$ by
\[
\rho_{n}\left(\begin{pmatrix}a & b\\
c & d
\end{pmatrix}\right)P\left(z_{1},z_{2}\right)=P\left(az_{1}+bz_{2},cz_{1}+dz_{2}\right),\quad P\in V_{n}.
\]
Then $\dim V_{n}=n+1$ and $\rho_{n}:\sl_{2}\left(\real\right)\rightarrow\gl\left(V_{n}\right)$
is a continuous representation of $\sl_{2}\left(\real\right)$. In
fact, $\rho_{n}$ is irreducible and all irreducible finite dimensional
representations of $\sl_{2}\left(\real\right)$ are obtained in this
way. The representation $\rho_{n}$ is called the representation of
$\sl_{2}\left(\real\right)$ of highest weight $n$. 

Recall that the Lie algebra $\mathfrak{sl}_{2}\left(\real\right)$
of $\sl_{2}\left(\real\right)$ is spanned by
\begin{equation}
\begin{matrix}h & =\begin{pmatrix}1 & 0\\
0 & -1
\end{pmatrix}, & e & =\begin{pmatrix}0 & 1\\
0 & 0
\end{pmatrix}, & f & =\begin{pmatrix}0 & 0\\
1 & 0
\end{pmatrix}\end{matrix}.\label{eq:sl2-relations}
\end{equation}
which satisfy the relations
\[
\begin{matrix}\left[h,e\right] & =2e, & \left[h,f\right] & =-2f, & \left[e,f\right] & =h\end{matrix}.
\]
A smooth finite dimensional representation $\rho:\sl_{2}\left(\real\right)\rightarrow\gl\left(V\right)$
of $\sl_{2}\left(\real\right)$ induces a representation of $\mathfrak{sl}_{2}\left(\real\right)$
by
\[
\rho\left(X\right).v=\frac{\d}{\d t}|_{t=0}\left(\exp\left(tX\right)\right).v,\quad v\in V,X\in\mathfrak{sl}_{2}\left(\real\right).
\]
On $V_{n}$, $\rho\left(h\right)$ acts as
\[
\rho\left(h\right)P\left(z_{1},z_{2}\right)=\frac{\d}{\d t}|_{t=0}\left(\begin{pmatrix}e^{t} & 0\\
0 & e^{-t}
\end{pmatrix}\right)P\left(z_{1},z_{2}\right)=\frac{\d}{\d t}|_{t=0}P\left(e^{t}z_{1},e^{-t}z_{2}\right).
\]
Therefore, with respect to the basis given by the monomials
\[
\left\{ z_{1}^{n},z_{1}^{n-1}z_{2},\dots,z_{2}^{n}\right\} ,
\]
we have:
\[
\rho_{n}\left(h\right)=\begin{pmatrix}n & 0 & 0 & 0\\
0 & n-2 & 0 & 0\\
0 & 0 & \ddots & 0\\
0 & 0 & 0 & -n
\end{pmatrix}.
\]
Hence, for $a_{t}\in\sl_{2}\left(\real\right)$, $t\ge0$, we have
\[
\rho_{n}\left(a_{t}\right)=\begin{pmatrix}e^{\frac{nt}{2}} & 0 & 0 & 0\\
0 & e^{\frac{\left(n-2\right)t}{2}} & 0 & 0\\
0 & 0 & \ddots & 0\\
0 & 0 & 0 & e^{-\frac{nt}{2}}
\end{pmatrix}.
\]

For a general finite dimensional complex representation of $\sl_{2}\left(\real\right)$
we have the following. 
\begin{thm}
(Weyl) Every finite dimensional complex representation of $\sl_{2}\left(\real\right)$
is completely reducible.
\end{thm}

\subsubsection{\label{subsec:Fin-dim-rep}Real finite dimensional representations
of $\protect\sl_{2}\left(\protect\real\right)$}

Let $V$ be a real vector space and $\rho:\sl_{2}\left(\real\right)\rightarrow\gl\left(V\right)$
a faithful representation. $\rho$ extends uniquely to a representation
$\rho_{\c}:\sl_{2}\left(\real\right)\rightarrow\gl\left(V\otimes_{\real}\c\right)$
of $\sl_{2}\left(\real\right)$ on the complexification $V_{\c}:=V\otimes_{\real}\c$
of $V$. By Weyl's theorem $\rho_{\c}$ decomposes to a direct sum
of irreducible representations:
\[
V_{\c}=\bigoplus_{i=1}^{n}V_{i}.
\]
The eigenvalues of $h$ on $V_{\c}$ are real. Therefore, we can find
highest weight vectors in each $V_{i}$ which are in $V$. An elementary
argument using the relations \ref{eq:sl2-relations} shows that in
fact each $V_{i}$ is stable under complex conjugation and $V\cap V_{i}$
is a real representation of $\sl_{2}\left(\real\right)$ with the
same highest weight as $V_{i}$. This shows:
\begin{prop}
Let $\rho:\sl_{2}\left(\real\right)\rightarrow\gl\left(V\right)$
be a faithful finite dimensional real representation of $\sl_{2}\left(\real\right)$
with $V_{\c}=\bigoplus_{i=1}^{n}V_{k_{i}}$ with $V_{k}$ an irreducible
representation of highest weight $k$. Then $V$ decomposes as a direct
sum of real irreducible representations with highest weights $k_{1},\dots,k_{n}$.
Let 
\[
A=\left\{ a_{t}=\begin{pmatrix}e^{\frac{t}{2}} & 0\\
0 & e^{-\frac{t}{2}}
\end{pmatrix}|t\in\real\right\} .
\]
There is a basis of $V$ in which $\rho\left(A\right)$ is simultaneously
diagonalizable with
\[
\rho\left(a_{t}\right)=\bigoplus\rho_{k_{i}}\left(a_{t}\right)
\]
where,
\[
\rho_{k}\left(a_{t}\right)=\begin{pmatrix}e^{\frac{kt}{2}} & 0 & 0 & 0\\
0 & e^{\frac{\left(k-2\right)t}{2}} & 0 & 0\\
0 & 0 & \ddots & 0\\
0 & 0 & 0 & e^{-\frac{kt}{2}}
\end{pmatrix}.
\]
\end{prop}

\subsubsection{\label{sec:K-invariant-norms}$K$-invariant norms}
\begin{defn}
Let $\rho:\sl_{2}\left(\real\right)\rightarrow\gl\left(V\right)$
be a faithful finite dimensional representation on a real vector space
$V$. An inner product $\left\langle .,.\right\rangle :V\times V\rightarrow\real$
on $V$ is called $\rho$\emph{-standard} if it is \emph{$\rho\left(K\right)$}-invariant\emph{
}and if there exists an orthonormal basis in which $\rho\left(A\right)$
is diagonal. If $\norm .:V\rightarrow\c$ is the norm induced by a
$\rho$-standard inner product on $V$, $\norm .$ is called a $\rho$-standard
norm on $V$.
\end{defn}
If $\norm .:V\rightarrow\real$ is $\rho$-standard, we can compute
the operator norm of the operators in $\rho\left(A^{+}\right)$. Indeed,
for $t\ge0$, $\rho\left(a_{t}\right)$ is orthogonally diagonalizable
with highest eigenvalue $e^{\frac{kt}{2}}$, where $k$ is the highest
weight appearing in $\rho$. Therefore,
\[
\norm{\rho\left(a_{t}\right)}_{\mathrm{op}}=\sup_{0\neq x\in V}\frac{\norm{\rho\left(a_{t}\right)x}}{\norm x}=e^{\frac{kt}{2}}.
\]

\begin{prop}
For any finite dimensional real representation $\rho$ there exist
a $\rho$-standard inner product on $V$.
\end{prop}
\begin{proof}
First, assume $V=V_{k}$ is irreducible with highest weight vector
$v_{0}\in V$. We determine an inner product $\left\langle .,.\right\rangle :V\times V\rightarrow\real$
on $V$ by setting the basis
\[
\left\{ v_{0},v_{1},\dots,v_{k}\right\} 
\]
with
\begin{align*}
\rho\left(f\right)v_{i} & =c_{i}v_{i+1}\\
c_{i} & =\sqrt{\left(k-i\right)\left(i+1\right)},
\end{align*}
to be orthonormal. With respect to this basis, the operators $\rho\left(a_{t}\right)$
are diagonal. Thus, to show that $\left\langle .,.\right\rangle $
is $\rho$-standard it remains to show that it is $\rho\left(K\right)$
invariant. 

$K$ is the one parameter group generated by the element $e-f\in\mathfrak{sl}_{2}\left(\real\right)$.
To show that the inner product is $\rho\left(K\right)$-invariant,
it is enough to show that $\rho\left(e-f\right)$ is antisymmetric
with respect to this inner product. Since $\rho\left(e\right)$ must
send a vector of weight $k-2i$ to a vector of weight $k-2i+2$, one
can verify by using $\left[\rho\left(e\right),\rho\left(f\right)\right]=\rho\left(h\right)$,
that
\[
\rho\left(e\right)v_{i}=d_{i}v_{i-1}
\]
with
\[
d_{i}=\sqrt{i\left(k-i+1\right)}.
\]
For all $0\le i,j\le1$, we have
\[
\left\langle \rho\left(e\right)v_{i},v_{j}\right\rangle =d_{i}\left\langle v_{i-1},v_{j}\right\rangle =d_{i}\delta_{i-1,j}.
\]
Note that
\[
c_{i-1}=\sqrt{\left(k-\left(i-1\right)\right)\left(\left(i-1\right)+1\right)}=\sqrt{\left(k-i+1\right)i}=d_{i}.
\]
Therefore,
\[
\left\langle v_{i},\rho\left(f\right)v_{j}\right\rangle =c_{j}\left\langle v_{i},v_{j+1}\right\rangle =c_{j}\delta_{i,j+1}=c_{i-1}\delta_{i-1,j}=d_{i}\delta_{i-1,j}.
\]
Hence, for any $0\le i,j\le k$,
\[
\left\langle \rho\left(e\right)v_{i},v_{j}\right\rangle =\left\langle v_{i},\rho\left(f\right)v_{j}\right\rangle .
\]
This means that $\rho\left(e\right)$ and $\rho\left(f\right)$ are
adjoints with respect to $\left\langle .,.\right\rangle $. Consequently,
the operator 
\[
\rho\left(e\right)-\rho\left(f\right)=\rho\left(e-f\right)
\]
is antisymmetric. Thus, for $V_{k}$ irreducible, $\rho$-standard
norms exist.

If $V$ is not irreducible, $V$ decomposes as a direct sum $V=\bigoplus_{i}V_{k_{i}}$
of irreducible representations, and we can take an inner product which
is the direct sum of $\rho$-standard inner products on each $V_{k_{i}}$.
\end{proof}

\subsection{\label{sec:Sector-estimates}Sector estimates for Fuchsian groups}

Let $\Gamma<\psl_{2}\left(\real\right)$ be a non-elementary, geometrically
finite Fuchsian group with critical exponent $\delta$. As in the
case of $\sl_{2}\left(\real\right)$, for $g\in\psl_{2}\left(\real\right)$
we define the coordinates $t\left(g\right),\theta_{1}\left(g\right),\theta_{2}\left(g\right)$
as the values $t\left(g\right)\ge0$ and $0\le\theta_{1}\left(g\right),\theta_{2}\left(g\right)<\pi$
such that $g=\pm k_{\theta_{1}\left(g\right)}a_{t\left(g\right)}k_{\theta_{2}\left(g\right)}$.
As mentioned in the introduction, a key step in the proof of our main
theorem involves a decomposition of functions on $\psl\left(2,\real\right)$
into harmonics with respect to the right and left $K$-types. This
will lead us to consider expressions of the form
\begin{equation}
\sum_{\gamma\in\Gamma,\norm{\gamma}\le T}e^{2in\theta_{1}\left(\gamma\right)}e^{2im\theta_{2}\left(\gamma\right)}.\label{eq:trigo-average-chap2}
\end{equation}
These were analyzed by J. Bourgain, A. Kontorovich, and P. Sarnak
in \cite{bourgain2010sector}. A similar result in the setting of $\mathbb{SL}_{2}(\mathbb{C})$ was proved by Vinogradov \cite{Vinogradov_2013}. The general statement for $\mathrm{SO}(n,1)$ under the existence of spectral gaps for $\Gamma$ is \cite{Mohammadi_2015}. We will only consider the case of $\mathrm{SL}_{2}(\mathbb{R})$ so we will only review here the results of \cite{bourgain2010sector}.

For lattices, the asymptotic distribution of such harmonics was studies
by Good \cite{good2006local}. Bourgain, Kontorovich, and Sarnak \cite{bourgain2010sector}
have managed to provide estimates in the infinite volume case. There
are two main differences between finite and infinite volume. The first
is that for a general discrete group of infinite co-volume, the eigenvector
with the lowest eigenvalue is no longer constant. If the critical
exponent $\delta$ satisfies $\delta>\frac{1}{2}$ and $\Gamma$ is
geometrically finite, the lowest eigenvalue is
\[
\lambda_{0}=\delta\left(1-\delta\right),
\]
and it is of multiplicity $1$. As was proven by Patterson in \cite[theorem 3.1]{patterson1976limit},
a corresponding eigenvector is given in terms of the Patterson-Sullivan
measure by
\[
\phi_{0}\left(x\right)=\int1\,\d\mu_{x}=\int P\left(x,\xi\right)^{\delta}\,\d\mu_{0},
\]
where $P\left(x,\xi\right)$ is the Poisson kernel. Therefore, the
main term of the limit 
\[
\sum_{\gamma\in\Gamma,\norm{\gamma}\le T}e^{2in\theta_{1}\left(\gamma\right)}e^{2\pi i\theta_{2}\left(\gamma\right)}
\]
will not distribute uniformly on the circle as in the case of a lattice.
In fact, the angular distribution is given by the Patterson-Sullivan
measure. Recall that choosing the angles $\theta_{1},\theta_{2}$
to lie between $0$ and $\pi$ means that the boundary of the hyperbolic
plane is identified with $\left[0,\pi\right]$ (see remark \ref{rem:symmetric}).
In this convention, using Polar coordinates $\left(r,\theta\right)$
on the unit disc with $0\le\theta<\pi$, the function $\phi_{0}$
can be written as:
\[
\phi_{0}\left(r,\theta\right)=\int_{0}^{\pi}\left(\frac{1-r^{2}}{\left|re^{2i\theta_{1}}-e^{2i\alpha}\right|^{2}}\right)\,\d\mu\left(\alpha\right),
\]
with $\mu$ as in \ref{rem:symmetric}. The lift of $\phi_{0}$ to
a function on $\psl_{2}\left(\real\right)$ can be expressed in the
coordinates $\left(\theta_{1},r,\theta_{2}\right)$ with $r\left(g\right)=\tanh\left(t\left(g\right)\right)$,
as
\[
\phi_{0}\left(\theta_{1},r,\theta_{2}\right)=\int_{0}^{\pi}\left(\frac{1-r^{2}}{\left|re^{2i\theta_{1}}-e^{2i\alpha}\right|^{2}}\right)\,\d\mu\left(\alpha\right).
\]

The eigenfunction $\phi_{0}$ then gives the leading term in theorem
\ref{thm:BKS}. The precise result is the following (\cite[theorem 1.5]{bourgain2010sector}).
\begin{thm}
(J. Bourgain, A. Kontorovich, P. Sarnak)\label{thm:BKS} Let $\Gamma$
be a non-elementary geometrically finite Fuchsian group of the second
kind with critical exponent $\delta>\frac{1}{2}$. Let
\[
0<\delta\left(1-\delta\right)=\lambda_{0}<\lambda_{1}<\dots<\lambda_{N}<\frac{1}{4}
\]
be the eigenvalues of the Laplacian on $\Gamma\backslash\mathbb{H}$
below $\frac{1}{4}$. Write
\[
\lambda_{i}=s_{i}\left(1-s_{i}\right)
\]
with $s_{i}>\frac{1}{2}$. Then, for integers $n$ and $k$ there
are constants $c_{1}\left(n,k\right),\dots,c_{N}\left(n,k\right)\in\c$
such that
\begin{align*}
\sum_{\gamma\in\Gamma_{T}}e^{2in\theta_{1}\left(\gamma\right)}e^{2ik\theta_{2}\left(\gamma\right)} & =\hat{\mu}\left(2n\right)\hat{\mu}\left(2k\right)\pi^{\frac{1}{2}}\frac{\Gamma\left(\delta-\nicefrac{1}{2}\right)}{\Gamma\left(\delta+1\right)}T^{2\delta}+\sum_{i=1}^{N}c_{i}\left(n,k\right)T^{2s_{i}}\\
 & +\o\left(T^{\frac{1}{4}+2\delta\cdot\frac{3}{4}}\log\left(T\right)^{\frac{1}{4}}\left(1+\left|n\right|+\left|k\right|\right)^{\frac{3}{4}}\right),
\end{align*}
as $T\rightarrow\infty$. Here $\left|c_{j}\left(n,k\right)\right|\ll\left|c_{i}\left(0,0\right)\right|$
as $n$ and $k$ vary, and the implied constants depend only on $\Gamma$.
\end{thm}

\section{Proof of the main theorem}

\subsection{Reduction to the case of functions which are invariant under $\rho\left(-1\right)$}

In order to work with $\psl_{2}\left(\real\right)$ instead of $\sl_{2}\left(\real\right)$,
we reduce to the case of functions which are invariant respect to
translation by $\rho\left(-1\right)$. That is, functions with satisfy
$f\left(\rho\left(-1\right)x\right)=f\left(x\right)$ for all $x\in\endd\left(\real^{2}\right)$.
Note that if $V$ decomposes as
\[
V=\bigoplus_{i}V_{k_{i}}
\]
into irreducible representations of $\sl_{2}\left(\real\right)$,
we have
\[
\rho\left(-1\right)=\bigoplus_{i}\rho_{k_{i}}\left(-1\right),
\]
with $\rho_{k}\left(-1\right)$ given by
\[
\rho_{k}\left(-1\right)=\begin{cases}
1 & k\text{ is even}\\
-1 & k\text{ is odd}
\end{cases}.
\]
Since $\norm{\rho\left(-1\right)}=1$ and $-1$ lies in the center
of $\sl_{2}\left(\real\right)$, summation over $\Gamma$ or over
$\Gamma_{T}$ for $T>0$ is invariant under multiplication by $\rho\left(-1\right)$.
Therefore, if $h\in C_{c}\left(\endd\left(V\right)\right)$ is satisfies
$h\left(\rho\left(-1\right)x\right)=-h\left(x\right)$ for all $x\in\endd\left(\real^{2}\right)$,
we have
\[
\mu_{\Gamma,T}\left(h\right)=\frac{1}{T^{\nicefrac{2\delta}{k}}}\sum_{\gamma\in\Gamma}h\left(\frac{\rho\left(\gamma\right)}{T}\right)=\frac{1}{T^{\frac{2\delta}{k}}}\sum_{\gamma\in\Gamma}h\left(\rho\left(-1\right)\frac{\rho\left(\gamma\right)}{T}\right)=-\frac{1}{T^{\frac{2\delta}{k}}}\sum_{\gamma\in\Gamma}h\left(\frac{\rho\left(\gamma\right)}{T}\right)
\]
so $\mu_{\Gamma,T}\left(h\right)=0$. Let $f$ be a continuous function
on $\endd\left(V\right)$. Then,
\begin{align*}
\mu_{\Gamma,T}\left(f\right) & =\frac{1}{T^{\nicefrac{2\delta}{k}}}\sum_{\gamma\in\Gamma}\frac{f\left(\frac{\rho\left(\gamma\right)}{T}\right)+f\left(\rho\left(-1\right)\frac{\rho\left(\gamma\right)}{T}\right)}{2}+\frac{f\left(\frac{\rho\left(\gamma\right)}{T}\right)-f\left(\rho\left(-1\right)\frac{\rho\left(\gamma\right)}{T}\right)}{2}\\
 & =\frac{1}{T^{\nicefrac{2\delta}{k}}}\sum_{\gamma\in\Gamma}\frac{f\left(\frac{\rho\left(\gamma\right)}{T}\right)+f\left(\rho\left(-1\right)\frac{\rho\left(\gamma\right)}{T}\right)}{2}.
\end{align*}
Assuming theorem \ref{thm:unnorm} is proved for functions which are
even with respect to $\rho$, we get
\begin{align*}
\lim_{T\rightarrow\infty}\mu_{\Gamma,T}\left(f\right) & =\frac{\delta}{2k}\cdot V_{\Gamma}\cdot\frac{1}{2}\cdot\int_{K\times\left[0,\infty\right)\times K}f\left(\rho\left(k_{\theta_{1}}\right)\cdot tP_{k}\cdot\rho\left(k_{\theta_{2}}\right)\right)t^{\frac{2\delta}{k}-1}\,\d\mu\left(\theta_{1}\right)\d t\d\mu\left(\theta_{2}\right)\\
 & +\frac{\delta}{2k}V_{\Gamma}\cdot\frac{1}{2}\cdot\int_{K\times\left[0,\infty\right)\times K}f\left(\rho\left(-1\right)\rho\left(k_{\theta_{1}}\right)\cdot tP_{k}\cdot\rho\left(k_{\theta_{2}}\right)\right)t^{\frac{2\delta}{k}-1}\,\d\mu\left(\theta_{1}\right)\d t\d\mu\left(\theta_{2}\right).
\end{align*}
Since $\mu$ is symmetric with respect to the transformation $\theta\mapsto\theta+\pi$
and
\[
\rho\left(-1\right)\rho\left(k_{\theta_{1}}\right)=\rho\left(k_{\theta_{1}}\right)=\rho\left(k_{\theta_{1}+\pi}\right),\quad\theta_{1}\in\left[0,2\pi\right]
\]
(see remark \ref{rem:symmetric}), we get:
\begin{align*}
\lim_{T\rightarrow\infty}\nu_{\Gamma,T}\left(f\right) & =\frac{\delta}{2k}V_{\Gamma}\cdot\frac{1}{2}\cdot\int_{K\times\left[0,\infty\right)\times K}f\left(\rho\left(k_{1}\right)\cdot tP_{k}\cdot\rho\left(k_{2}\right)\right)t^{\frac{2\delta}{k}-1}\,\d\mu\left(k_{1}\right)\d t\d\mu\left(k_{2}\right)\\
 & +\frac{\delta}{2k}V_{\Gamma}\cdot\frac{1}{2}\cdot\int_{K\times\left[0,\infty\right)\times K}f\left(\rho\left(k_{1}\right)\cdot tP_{k}\cdot\rho\left(k_{2}\right)\right)t^{\frac{2\delta}{k}-1}\,\d\mu\left(k_{1}\right)\d t\d\mu\left(k_{2}\right).\\
 & =\mu_{\Gamma}\left(f\right).
\end{align*}

Clearly, the same argument also works for the the quantitative theorem
\ref{thm:Main_Quant}. Thus, to prove theorems \ref{thm:unnorm},
\ref{thm:normalized}, and \ref{thm:Main_Quant} it suffices to prove
theorem \ref{thm:normalized} and section (2) of theorem \ref{thm:Main_Quant}
for functions that are even under the action of $\rho\left(-1\right)$.
The rest of the chapter is devoted to the proof of these.

\subsection{\label{sec:Partition-into-annuli}Partition of the sum into annuli}

Let $f:\endd\left(V\right)\rightarrow\c$ be continuous. We will analyze
the averages $\nu_{T}\left(f\right)$ by partitioning the sum into
radial annuli. Recall that our fixed norm on $\endd\left(V\right)$
is an operator norm. Therefore, for any $x,y\in\endd\left(\real^{2}\right)$,
\[
\norm{xy}\le\norm x\norm y.
\]
Recall that we have defined
\[
\nu_{T}\left(f\right)=\frac{1}{T^{\frac{2\delta}{k}}}\sum_{\gamma\in\Gamma_{T}}f\left(\frac{\rho\left(\gamma\right)}{T}\right).
\]
We will partition the sum $\nu_{T}\left(f\right)$ into radial sections
according to the norm of the elements in $\rho\left(\Gamma\right)$.
For $N\in\nat$ write
\[
\nu_{T}\left(f\right)=\frac{1}{T^{\frac{2\delta}{k}}}\sum_{j=0}^{N-1}\sum_{\frac{j}{N}T<\norm{\rho\left(\gamma\right)}\le\frac{j+1}{N}T}f\left(\frac{\rho\left(\gamma\right)}{T}\right).
\]
in order to simplify our notation, we define
\[
S_{T,j}=\left\{ \gamma\in\Gamma|\frac{j}{N}T<\norm{\rho\left(\gamma\right)}\le\frac{j+1}{N}T\right\} .
\]
Then, with this notation,
\[
\nu_{T}\left(f\right)=\frac{1}{T^{\frac{2\delta}{k}}}\sum_{j=0}^{N-1}\sum_{\gamma\in S_{T,j}}f\left(\frac{\rho\left(\gamma\right)}{T}\right).
\]
Recall that the norm $\norm .$ is invariant under $\rho\left(K\right)$
and $\norm{\rho\left(a_{t}\right)}=e^{\frac{kt}{2}}$ for all $t\ge0$.
Hence, for $\gamma\in\Gamma$ with $\gamma=k_{1}a_{t}k_{2}$, we have
\[
\norm{\rho\left(\gamma\right)}=\norm{\rho\left(k_{1}\right)\rho\left(a_{t}\right)\rho\left(k_{2}\right)}=\norm{\rho\left(a_{t}\right)}=e^{\frac{kt}{2}}.
\]
Define, as in section \ref{sec:K-invariant-norms}, for $\gamma\in\Gamma$,
\[
\left|\gamma\right|=e^{\frac{1}{2}\d\left(o,\gamma.o\right)}.
\]
where $\d$ is the hyperbolic distance on $\mathbb{H}^{2}$ and $o\in\mathbb{H}^{2}$
is stabilized by $K$. Then, 
\[
S_{T,j}=\left\{ \gamma\in\Gamma|\left(\frac{j}{N}T\right)^{\frac{1}{k}}<\left|\gamma\right|\le\left(\frac{j+1}{N}T\right)^{\frac{1}{k}}\right\} .
\]
The set $S_{T,j}.o\subseteq\mathbb{H}^{2}$ is then an annulus in
hyperbolic space with center $o$ and radii $\left(\frac{j}{N}T\right)^{\frac{1}{k}}$
and $\left(\frac{j+1}{N}T\right)^{\frac{1}{k}}$. Lax and Phillips
\cite[theorem 1]{lax1982asymptotic} have shown that
\[
\lim_{T\rightarrow\infty}\frac{\#\left\{ \gamma\in\Gamma|\d\left(o,\gamma.o\right)<R\right\} }{e^{\delta R}}=V_{\Gamma}.
\]
Therefore,
\begin{equation}
\lim_{T\rightarrow\infty}\frac{\left|S_{T,j}\right|}{T^{\frac{2\delta}{k}}}=\lim_{T\rightarrow\infty}\frac{1}{T^{\frac{2\delta}{k}}}\left(\left|\Gamma_{\left(\frac{j+1}{N}T\right)^{\frac{2}{k}}}\right|-\left|\Gamma_{\left(\frac{j}{N}T\right)^{\frac{2}{k}}}\right|\right)=\frac{V_{\Gamma}}{N^{\frac{2\delta}{k}}}\left(\left(j+1\right)^{\frac{2\delta}{k}}-j^{\frac{2\delta}{k}}\right).\label{eq:M-Tj-is-thenumber}
\end{equation}
The idea of the proof of our main theorems is to replace $\frac{\rho\left(\gamma\right)}{T}$
in the argument of $f$ with $\frac{j}{N}\frac{\rho\left(\gamma\right)}{\norm{\rho\left(\gamma\right)}}$.
As we shall see, this can be done when $N$ is large enough.

Fix $\epsilon>0$. We will show that there is a $T\left(\epsilon\right)$
such that for $T>T\left(\epsilon\right)$,
\[
\left|\nu_{T}\left(f\right)-\nu_{\Gamma}\left(f\right)\right|<\epsilon.
\]
The function $f$ is continuous in the closed unit ball in $\endd\left(V\right)$,
which is compact. Let $N_{0}$ be such that for $N>N_{0}$, whenever
$\norm{x-y}<\frac{1}{N}$ with $x,y\in\endd\left(V\right)$ of norm
$\norm x,\norm y\le1$ we have,
\[
\left|f\left(x\right)-f\left(y\right)\right|<\epsilon.
\]
By the definition of $S_{T,j}$, $\frac{j}{N}<\frac{\norm{\rho\left(\gamma\right)}}{T}\le\frac{j+1}{N}$
for all $\gamma\in S_{T,j}$. Thus, for $\gamma\in S_{T,j}$,
\[
\norm{\frac{\rho\left(\gamma\right)}{T}-\frac{j}{N}\frac{\rho\left(\gamma\right)}{\norm{\rho\left(\gamma\right)}}}=\norm{\frac{\rho\left(\gamma\right)}{\norm{\rho\left(\gamma\right)}}\left(\frac{\norm{\rho\left(\gamma\right)}}{T}-\frac{j}{N}\right)}\le\left|\frac{\norm{\rho\left(\gamma\right)}}{T}-\frac{j}{N}\right|\cdot\frac{\norm{\rho\left(\gamma\right)}}{\norm{\rho\left(\gamma\right)}}<\frac{1}{N}.
\]
We will use this estimate to replace $\frac{\rho\left(\gamma\right)}{T}$
with $\frac{j}{N}\frac{\rho\left(\gamma\right)}{\norm{\rho\left(\gamma\right)}}$
when summing over $\gamma$ in the annulus $S_{T,j}$. Define
\[
\tilde{\nu}_{T,N}\left(f\right)=\frac{1}{T^{\frac{2\delta}{k}}}\sum_{j=0}^{N-1}\sum_{\gamma\in S_{T,j}}f\left(\frac{j}{N}\frac{\rho\left(\gamma\right)}{\norm{\rho\left(\gamma\right)}}\right).
\]
Then, for any $T>0$,
\begin{align*}
\left|\nu_{T}\left(f\right)-\tilde{\nu}_{T,N}\left(f\right)\right| & \le\frac{1}{T^{\frac{2\delta}{k}}}\sum_{j=0}^{N-1}\sum_{\gamma\in S_{T,j}}\left|f\left(\frac{\rho\left(\gamma\right)}{T}\right)-f\left(\frac{j}{N}\frac{\rho\left(\gamma\right)}{\norm{\rho\left(\gamma\right)}}\right)\right|\le\frac{\left|\Gamma_{T}\right|}{T^{\frac{2\delta}{k}}}\epsilon.
\end{align*}
Since $\frac{\left|\Gamma_{T}\right|}{T^{\frac{2\delta}{k}}},\frac{\left|S_{T,j}\right|}{T^{\frac{2\delta}{k}}}$
converges, $\frac{\left|\Gamma_{T}\right|}{T^{\frac{2\delta}{k}}}$
and $\frac{\left|S_{T,j}\right|}{T^{\frac{2\delta}{k}}}$ are bounded
for all $T\ge0$. Let $C_{\Gamma}$ denote a bound for both. Then,
for $N>N_{0}$,
\begin{equation}
\left|\nu_{T}\left(f\right)-\tilde{\nu}_{T,N}\left(f\right)\right|\le C_{\Gamma}\epsilon.\label{eq:par_bound_1}
\end{equation}
We can therefore deal with the averages $\tilde{\nu}_{T,N}\left(f\right)$
instead of $\nu_{T}\left(f\right)$. 

We wish to analyze the sums over the annuli $S_{T,j}$ individually.
In order to accomplish that we need to normalize the sum over each
$S_{T,j}$. Define:
\[
M_{T,j}=\left(\frac{T}{N}\right)^{\frac{2\delta}{k}}\left(\left(j+1\right)^{\frac{2\delta}{k}}-j^{\frac{2\delta}{k}}\right).
\]
From \ref{eq:M-Tj-is-thenumber} we see that $M_{T,j}\sim\left|S_{T,j}\right|$
as $T$ tends to infinity. We can write $\tilde{\nu}_{T,N}\left(f\right)$
as
\begin{align*}
\tilde{\nu}_{T,N}\left(f\right) & =\frac{1}{N^{\frac{2\delta}{k}}}\cdot\frac{1}{\left(\frac{T}{N}\right)^{\frac{2\delta}{k}}}\sum_{\norm{\gamma}\le\frac{T}{N}}f\left(\frac{j}{N}\frac{\rho\left(\gamma\right)}{\norm{\rho\left(\gamma\right)}}\right)\\
 & +\sum_{j=1}^{N-1}\left(\frac{1}{N}\right)^{\frac{2\delta}{k}}\left(\left(j+1\right)^{\frac{2\delta}{k}}-j^{\frac{2\delta}{k}}\right)\frac{1}{M_{T,j}}\sum_{\gamma\in S_{T,j}}f\left(\frac{j}{N}\frac{\rho\left(\gamma\right)}{\norm{\rho\left(\gamma\right)}}\right).
\end{align*}
Define
\[
\nu_{T,N}\left(f\right)=\frac{2\delta}{k}\sum_{j=1}^{N-1}\left(\frac{1}{N}\right)^{\frac{2\delta}{k}}j^{\frac{2\delta}{k}-1}\frac{1}{M_{T,j}}\sum_{\gamma\in S_{T,j}}f\left(\frac{j}{N}\frac{\rho\left(\gamma\right)}{\norm{\rho\left(\gamma\right)}}\right).
\]
In the average $\nu_{T,N}$, the sum over each annulus is normalized.
The difference between $\nu_{T,N}\left(f\right)$ and $\tilde{\nu}_{T,N}\left(f\right)$
is then bounded by
\begin{align*}
\left|\tilde{\nu}_{T,N}\left(f\right)-\nu_{T,N}\left(f\right)\right| & \le\norm f_{\infty,1}\cdot\frac{1}{N^{\frac{2\delta}{k}}}\left(\frac{N}{T}\right)^{\frac{2\delta}{k}}\left|\Gamma_{\frac{T}{N}}\right|\\
 & +\norm f_{\infty,1}\sum_{j=1}^{N-1}\left(\frac{1}{N}\right)^{\frac{2\delta}{k}}j^{\frac{2\delta}{k}-1}\left|\frac{\left(j+1\right)^{\frac{2\delta}{k}}-j^{\frac{2\delta}{k}}}{j^{\frac{2\delta}{k}-1}}-\frac{2\delta}{k}\right|\frac{\left|S_{T,j}\right|}{M_{T,j}},
\end{align*}
with $\norm f_{\infty,1}=\max_{\norm x\le1}\left|f\left(x\right)\right|$.
As $j$ ranges over integers between $0$ and $N$, if $j\neq0$ $\frac{1}{j}$
is always smaller than $1$. Hence, if $j\neq0$ there is some constant
$C$, which does not depend on $j$, such that,
\begin{equation}
\left|\left(j+1\right)^{\frac{2\delta}{k}}-j^{\frac{2\delta}{k}}-\frac{2\delta}{k}j^{\frac{2\delta}{k}-1}\right|\le C\cdot j^{\frac{2\delta}{k}-2}.\label{eq:(j+1)-j}
\end{equation}
In particular, 
\[
\lim_{j\rightarrow\infty}\frac{\left(j+1\right)^{\frac{2\delta}{k}}-j^{\frac{2\delta}{k}}}{j^{\frac{2\delta}{k}-1}}=\frac{2\delta}{k}.
\]
Let $N_{1}>N_{0}>0$ be such that for $N>N_{1}$, if $j>\epsilon N$,
we have
\[
\left|\frac{\left(j+1\right)^{\frac{2\delta}{k}}-j^{\frac{2\delta}{k}}}{j^{\frac{2\delta}{k}-1}}-\frac{2\delta}{k}\right|<\epsilon.
\]
We can write $\left|\tilde{\nu}_{T,N}\left(f\right)-\nu_{T,N}\left(f\right)\right|$
as
\begin{align*}
\left|\tilde{\nu}_{T,N}\left(f\right)-\nu_{T,N}\left(f\right)\right| & \le\norm f_{\infty,1}\cdot\frac{1}{N^{\frac{2\delta}{k}}}\left(\frac{N}{T}\right)^{\frac{2\delta}{k}}\left|\Gamma_{\frac{T}{N}}\right|\\
 & +\norm f_{\infty,1}\sum_{j=1}^{\left\lfloor \epsilon N\right\rfloor }\left(\frac{1}{N}\right)^{\frac{2\delta}{k}}j^{\frac{2\delta}{k}-1}\left|\frac{\left(j+1\right)^{\frac{2\delta}{k}}-j^{\frac{2\delta}{k}}}{j^{\frac{2\delta}{k}-1}}-\frac{2\delta}{k}\right|\frac{\left|S_{T,j}\right|}{M_{T,j}}\\
 & +\norm f_{\infty,1}\sum_{j=\left\lfloor \epsilon N\right\rfloor +1}^{N-1}\left(\frac{1}{N}\right)^{\frac{2\delta}{k}}j^{\frac{2\delta}{k}-1}\left|\frac{\left(j+1\right)^{\frac{2\delta}{k}}-j^{\frac{2\delta}{k}}}{j^{\frac{2\delta}{k}-1}}-\frac{2\delta}{k}\right|\frac{\left|S_{T,j}\right|}{M_{T,j}}.
\end{align*}
where $\left\lfloor \epsilon N\right\rfloor $ is the largest integer
which is smaller than $\epsilon N$. Then, for $N>N_{2}=\max\left\{ N_{1},\epsilon^{-\frac{k}{2\delta}}\right\} $,
and $j>\epsilon N$,
\begin{align*}
\left(\frac{N}{T}\right)^{\frac{2\delta}{k}}\left|\Gamma_{\frac{T}{N}}\right| & \le C_{\Gamma}\\
\frac{1}{N^{\frac{2\delta}{k}}} & <\epsilon.\\
\frac{\left|S_{T,j}\right|}{M_{T,j}} & \le C_{\Gamma}\\
\left|\frac{\left(j+1\right)^{\frac{2\delta}{k}}-j^{\frac{2\delta}{k}}}{j^{\frac{2\delta}{k}-1}}-\frac{2\delta}{k}\right| & <\epsilon.
\end{align*}
Also, the sequence
\[
\frac{\left(j+1\right)^{\frac{2\delta}{k}}-j^{\frac{2\delta}{k}}}{j^{\frac{2\delta}{k}-1}}-\frac{2\delta}{k}
\]
is bounded by some $S>0$ for all $j\in\nat$. In total, we get
\begin{align*}
\left|\tilde{\nu}_{T,N}\left(f\right)-\nu_{T,N}\left(f\right)\right| & \le\norm f_{\infty,1}\frac{1}{N^{\frac{2\delta}{k}}}C_{\Gamma}+\norm f_{\infty,1}\sum_{j=1}^{\left\lfloor \epsilon N\right\rfloor }\left(\frac{1}{N}\right)^{\frac{2\delta}{k}}j^{\frac{2\delta}{k}-1}\cdot S\cdot C_{\Gamma}\\
 & +\norm f_{\infty,1}\sum_{j=1}^{N-1}\left(\frac{1}{N}\right)^{\frac{2\delta}{k}}j^{\frac{2\delta}{k}-1}C_{\Gamma}\epsilon.
\end{align*}
For $L>0$, the sum $\sum_{j=1}^{LN}\left(\frac{1}{N}\right)^{\frac{2\delta}{k}}j^{\frac{2\delta}{k}-1}$
can be written as $\sum_{j=1}^{LN}\frac{1}{N}\left(\frac{j}{N}\right)^{\frac{2\delta}{k}-1}$.
Therefore, it is bounded by the integral
\[
\int_{0}^{L}x^{\frac{2\delta}{k}-1}\,\d x=\frac{k}{2\delta}L^{\frac{2\delta}{k}}.
\]
Thus, for $N>N_{2}$ we get
\begin{align}
\left|\tilde{\nu}_{T,N}\left(f\right)-\nu_{T,N}\left(f\right)\right| & \le\norm f_{\infty,1}C_{\Gamma}\epsilon+S\norm f_{\infty,1}C_{\Gamma}\frac{k}{2\delta}\cdot\epsilon^{\frac{2\delta}{k}}+C_{\Gamma}\norm f_{\infty,1}\cdot\frac{k}{2\delta}\cdot\epsilon\nonumber \\
 & \le\left(1+\frac{k}{2\delta}\left(S+1\right)\right)C_{\Gamma}\norm f_{\infty,1}\left(\epsilon+\epsilon^{\frac{2\delta}{k}}\right).\label{eq:par_bound_2}
\end{align}
Combining \ref{eq:par_bound_1} and \ref{eq:par_bound_2} we see that
for $N>N_{2}$ and $T>0$, 
\begin{align*}
\left|\nu_{T}\left(f\right)-\nu_{T,N}\left(f\right)\right| & \le\left|\nu_{T}\left(f\right)-\tilde{\nu}_{T,N}\left(f\right)\right|+\left|\tilde{\nu}_{T,N}\left(f\right)-\nu_{T,N}\left(f\right)\right|.\\
 & \le\left(\left(1+\frac{k}{2\delta}\left(S+1\right)\right)C_{\Gamma}\norm f_{\infty,1}+C_{\Gamma}\right)\left(\epsilon+\epsilon^{\frac{2\delta}{k}}\right).
\end{align*}
Put $C_{0}=\left(1+\frac{k}{2\delta}\left(S+1\right)\right)C_{\Gamma}\norm f_{\infty,1}+C_{\Gamma}$.
Then, for $N>N_{2}$ and every $T>0$
\begin{equation}
\left|\nu_{T}\left(f\right)-\nu_{T,N}\left(f\right)\right|\le C_{0}\left(\epsilon+\epsilon^{\frac{2\delta}{k}}\right).\label{eq:SUM-TO-ANNULI}
\end{equation}
Thus, to analyze the measures $\nu_{T}\left(f\right)$ it is possible
to consider instead the measures $\nu_{T,N}$
\[
\nu_{T,N}\left(f\right)=\frac{2\delta}{k}\sum_{j=1}^{N-1}\left(\frac{1}{N}\right)^{\frac{2\delta}{k}}j^{\frac{2\delta}{k}-1}\frac{1}{M_{T,j}}\sum_{\gamma\in S_{T,j}}f\left(\frac{j}{N}\frac{\rho\left(\gamma\right)}{\norm{\rho\left(\gamma\right)}}\right).
\]

\begin{rem}
Note that the sum $\sum_{j=1}^{N}\frac{1}{N}\left(\frac{j}{N}\right)^{\frac{2\delta}{k}-1}$
looks like a ``Riemann sum'' of $\int_{0}^{1}t^{\frac{2\delta}{k}-1}\,\d t$
(this function is not necessarily bounded). This will be the source
of the integration over the interval $\left[0,1\right]$ in the expression
of the limit measure $\nu_{\Gamma}$.
\end{rem}
We now turn our attention to the averages over individual annuli.
Define:
\begin{equation}
\nu_{T,N,j}\left(f\right)=\frac{1}{M_{T,j}}\sum_{\gamma\in S_{T,j}}f\left(\frac{j}{N}\frac{\rho\left(\gamma\right)}{\norm{\rho\left(\gamma\right)}}\right).\label{eq:V_T-N-j}
\end{equation}
$\nu_{T,N,j}\left(f\right)$ is a normalized average over the annuli
$S_{T,j}$. We have
\[
\nu_{T,N}\left(f\right)=\frac{2\delta}{k}\sum_{j=1}^{N-1}\left(\frac{1}{N}\right)^{\frac{2\delta}{k}}j^{\frac{2\delta}{k}-1}\nu_{T,N,j}\left(f\right).
\]

In section \ref{subsec:Analysis-of-averages} we will analyze the
measures $\nu_{T,N,j}\left(f\right)$ as $T$ tends to infinity.

\subsubsection{Error term for H\"{o}lder functions}

In order to get the error bound in theorem \ref{thm:Main_Quant} we
need a quantitative estimate of the difference between $\nu_{T}\left(f\right)$
and $\nu_{T,N}\left(f\right)$. In this subsection we assume $f:\endd\left(V\right)\rightarrow\c$
is H\"{o}lder continuous with exponent $\alpha\in\left(0,1\right]$ and
constant $\norm f_{\mathrm{Lip}\alpha}$ with respect our fixed norm,
i.e.,
\[
\left|f\left(x\right)-f\left(y\right)\right|\le\norm f_{\mathrm{Lip}\alpha}\norm{x-y}^{\alpha},\quad x,y\in\endd\left(V\right).
\]
Recall that for $0\le j\le N$ and $\gamma\in S_{T,j}$,
\[
\norm{\frac{\rho\left(\gamma\right)}{T}-\frac{j}{N}\frac{\rho\left(\gamma\right)}{\norm{\rho\left(\gamma\right)}}}<\frac{1}{N}.
\]
Consequently,
\begin{align*}
\left|\nu_{T}\left(f\right)-\frac{1}{T^{\frac{2\delta}{k}}}\sum_{j=0}^{N}\sum_{\gamma\in S_{T,j}}f\left(\frac{j}{N}\frac{\rho\left(\gamma\right)}{\norm{\rho\left(\gamma\right)}}\right)\right| & \le\frac{1}{T^{\frac{2\delta}{k}}}\sum_{j=0}^{N-1}\sum_{\gamma\in S_{T,j}}\left|f\left(\frac{\rho\left(\gamma\right)}{T}\right)-f\left(\frac{j}{N}\frac{\rho\left(\gamma\right)}{\norm{\rho\left(\gamma\right)}}\right)\right|\\
 & \le\frac{1}{T^{\frac{2\delta}{k}}}\sum_{j=0}^{N-1}\sum_{\gamma\in S_{T,j}}\norm f_{\mathrm{Lip}\alpha}N^{-\alpha}=\norm f_{\mathrm{Lip}\alpha}\frac{\left|\Gamma_{T}\right|}{T^{\frac{2\delta}{k}}}N^{-\alpha}\\
 & \le\norm f_{\mathrm{Lip}\alpha}C_{\Gamma}N^{-\alpha}.
\end{align*}
We also need to estimate the error of replacing $\left(j+1\right)^{\frac{2\delta}{k}}-j^{\frac{2\delta}{k}}$
with $\frac{2\delta}{k}\cdot j^{\frac{2\delta}{k}-1}$. Recall from
\ref{eq:(j+1)-j} that for some constant $C>0$,
\[
\left|\left(j+1\right)^{\frac{2\delta}{k}}-j^{\frac{2\delta}{k}}-\frac{2\delta}{k}j^{\frac{2\delta}{k}-1}\right|\le Cj^{\frac{2\delta}{k}-2},
\]
for all $j\ge1$. Using this we can estimate the difference
\begin{align}
\left|\frac{1}{T^{\frac{2\delta}{k}}}\sum_{j=0}^{N-1}\sum_{\gamma\in S_{T,j}}f\left(\frac{j}{N}\frac{\rho\left(\gamma\right)}{\norm{\rho\left(\gamma\right)}}\right)-\frac{2\delta}{k}\sum_{j=1}^{N-1}\frac{1}{N}\left(\frac{j}{N}\right)^{\frac{2\delta}{k}-1}\frac{1}{M_{T,j}}\sum_{\gamma\in S_{T,j}}f\left(\frac{j}{N}\frac{\rho\left(\gamma\right)}{\norm{\rho\left(\gamma\right)}}\right)\right|.\label{eq:difference-annuli}
\end{align}
In the first sum, the term $j=0$ is given by
\[
\frac{1}{T^{\frac{2\delta}{k}}}\sum_{\gamma\in S_{T,0}}f\left(0\right)=\frac{\left|S_{T,0}\right|}{T^{\frac{2\delta}{k}}}f\left(0\right).
\]
This term is bounded by
\[
\norm f_{\infty,1}\frac{1}{N^{\frac{2\delta}{k}}}\frac{\left|\Gamma_{\nicefrac{T}{N}}\right|}{\left(\frac{T}{N}\right)^{\frac{2\delta}{k}}}\le C_{\Gamma}\frac{1}{N^{\frac{2\delta}{k}}}\norm f_{\infty,1}.
\]
The sum over the rest of the terms $1\le j\le N-1$ is bounded by
\[
\left|\sum_{j=1}^{N-1}\frac{1}{N^{\frac{2\delta}{k}}}\left(\left(j+1\right)^{\frac{2\delta}{k}}-j^{\frac{2\delta}{k}}\right)\frac{1}{M_{T,j}}\sum_{\gamma\in S_{T,j}}f\left(\frac{j}{N}\frac{\rho\left(\gamma\right)}{\norm{\rho\left(\gamma\right)}}\right)-\frac{2\delta}{k}\sum_{j=1}^{N-1}\frac{1}{N}\left(\frac{j}{N}\right)^{\frac{2\delta}{k}-1}\frac{1}{M_{T,j}}\sum_{\gamma\in S_{T,j}}f\left(\frac{j}{N}\frac{\rho\left(\gamma\right)}{\norm{\rho\left(\gamma\right)}}\right)\right|.
\]
Which is smaller or equal than the following:
\[
\sum_{j=1}^{N-1}\frac{1}{N^{\frac{2\delta}{k}}}\left|\left(j+1\right)^{\frac{2\delta}{k}}-j^{\frac{2\delta}{k}}-\frac{2\delta}{k}j^{\frac{2\delta}{k}-1}\right|\frac{\left|S_{T,j}\right|}{M_{T,j}}\norm f_{\infty,1}\cdot
\]
This is then bounded by
\[
CC_{\Gamma}\norm f_{\infty,1}\sum_{j=1}^{N-1}\frac{1}{N^{\frac{2\delta}{k}}}j^{\frac{2\delta}{k}-2}=CC_{\Gamma}\norm f_{\infty,1}\frac{1}{N^{\frac{2\delta}{k}}}\sum_{j=1}^{N-1}j^{\frac{2\delta}{k}-2}.
\]
For $\frac{1}{2}<\delta<1$, we always have $\frac{2\delta}{k}-2\neq0,-1$,
and so the sum $\sum_{j=1}^{N-1}j^{\frac{2\delta}{k}-2}$ is bounded
by the integral
\[
\sum_{j=1}^{N-1}j^{\frac{2\delta}{k}-2}\le\int_{1}^{N}t^{\frac{2\delta}{k}-2}\,\d t\le C_{\delta,k}N^{\frac{2\delta}{k}-1}.
\]
so, in this case, the difference \ref{eq:difference-annuli} is bounded
by
\[
C_{\Gamma}\norm f_{\infty,1}N^{-\frac{2\delta}{k}}+C_{\delta,k}C_{\Gamma}\norm f_{\infty,1}N^{-1}.
\]
If $\delta=1$ and $k>2$, the same bound holds. If $\delta=k=1$,
we have 
\[
\frac{1}{N^{\frac{2\delta}{k}}}\sum_{j=1}^{N-1}j^{\frac{2\delta}{k}-2}=\frac{1}{N^{2}}\sum_{j=1}^{N-1}j^{0}\le\frac{1}{N}
\]
Finally, If $\delta=1$ and $k=2$, then $\frac{2\delta}{k}-2=-1$
so
\[
\frac{1}{N}\sum_{j=1}^{N-1}j^{-1}\le\frac{1}{N}\log N.
\]
Hence, for any $\frac{1}{2}<\delta\le1$ and $k\in\nat$,
\[
\frac{1}{N^{\frac{2\delta}{k}}}\sum_{j=1}^{N-1}j^{\frac{2\delta}{k}-2}\le\tilde{C}_{1}\frac{1}{N}\log N,
\]
For some constant $\tilde{C}_{1}$ depending only on $\Gamma$ and
$k$. Define $C_{1}=CC_{\Gamma}\tilde{C}_{1}$. $C_{1}$ depends only
on $\Gamma$ and $k$, and in total,
\begin{equation}
\left|\nu_{T}\left(f\right)-\nu_{T,N}\left(f\right)\right|\le C_{1}\norm f_{\infty,1}\left(N^{-\frac{2\delta}{k}}+N^{-1}\log N\right)+\norm f_{\mathrm{Lip}\alpha}C_{\Gamma}N^{-\alpha}.\label{eq:Par-to-N-error}
\end{equation}

\subsection{Analysis of averages over annuli\label{subsec:Analysis-of-averages}}

We now turn to analyze the averages $\nu_{T,N,j}\left(f\right)$ over
the individual annuli. Recall that
\[
\nu_{T,N,j}\left(f\right)=\frac{1}{M_{T,j}}\sum_{\gamma\in S_{T,j}}f\left(\frac{j}{N}\frac{\rho\left(\gamma\right)}{\norm{\rho\left(\gamma\right)}}\right),\quad j=1,2,\dots,N-1.
\]
We will see that these averages converge as $T$ tends to infinity.
This will be done using the sector estimates obtained by Bourgain,
Kontorovich, and Sarnak (theorem \ref{thm:BKS}). 

To deal with all the averages $\nu_{T,N,j}\left(f\right)$ at once
and to simplify notation we will normalize the parameters $N$ and
$j$. The general setting is the following. Fix $\beta\in\left(0,1\right)$
and define:
\begin{align*}
S_{T} & =\left\{ \gamma\in\Gamma|T<\norm{\rho\left(\gamma\right)}\le\left(1+\beta\right)T\right\} \\
M_{T} & =T^{\frac{2\delta}{k}}\left(\left(1+\beta\right)^{\frac{2\delta}{k}}-1\right).
\end{align*}
In the case of the measures $\nu_{T,N,j}$, $T$ is replaced by $\frac{Tj}{N}$
and $\beta$ by $\frac{1}{j}$. Fix a constant $0<a\le1$ ($a$ plays
the role of $\frac{j}{N}$). We will consider the measures $\lambda_{T}$
on $\endd\left(V\right)$ given by:
\[
\lambda_{T}\left(f\right)=\frac{1}{M_{T}}\sum_{\gamma\in S_{T}}f\left(a\cdot\frac{\rho\left(\gamma\right)}{\norm{\rho\left(\gamma\right)}}\right),\quad f\in C\left(\endd\left(V\right)\right).
\]
Define a function $h:\psl_{2}\left(\real\right)\rightarrow\c$ by
\begin{equation}
h\left(g\right)=f\left(a\cdot\frac{\rho\left(\tilde{g}\right)}{\norm{\rho\left(\tilde{g}\right)}}\right),\quad g\in\psl_{2}\left(\real\right)\label{eq:definition-of-h}
\end{equation}
where $\tilde{g}\in\sl_{2}\left(\real\right)$ has image $g$ in $\psl_{2}\left(\real\right)$.
$h$ is well defined since by assumption $f$ is invariant under translation
by $\rho\left(-1\right)$. Let $\overline{S}_{T}$ be the image of
$S_{T}$ in $\psl_{2}\left(\real\right)$. Since $\Gamma$ contains
$-1$,
\begin{equation}
\lambda_{T}\left(f\right)=\frac{1}{M_{T}}\sum_{\gamma\in S_{T}}f\left(a\cdot\frac{\rho\left(\gamma\right)}{\norm{\rho\left(\gamma\right)}}\right)=\frac{2}{M_{T}}\sum_{\gamma\in\overline{S}_{T}}h\left(\gamma\right).\label{eq:alpha-h}
\end{equation}
Let $\theta_{1}\left(g\right),\theta_{2}\left(g\right)$ be the Cartan
coordinates on $\psl_{2}\left(\real\right)$. Instead of the standard
coordinate $t\left(g\right)$ on $A^{+}$, it will be more useful
to us to use the coordinate $r\left(g\right)$ obtained from $t$
by $r=\tanh\left(\frac{t}{2}\right)$ so that $r\in\left[0,1\right)$.
\begin{lem}
\label{lem:extension-of-h}For any $0\le\theta_{1},\theta_{2}\le\pi$
the limit
\[
\lim_{r\rightarrow1^{-}}h\left(k\left(\theta_{1}\right)a_{t\left(r\right)}k\left(\theta_{2}\right)\right)
\]
exists and is uniform in $\theta_{1},\theta_{2}$. Consequently, the
function $\tilde{h}:\left[0,\pi\right]\times\left[0,1\right)\times\left[0,\pi\right]\rightarrow\c$
given by
\[
\tilde{h}\left(\theta_{1},r,\theta_{2}\right)=h\left(k_{\theta_{1}}a_{t\left(r\right)}k_{\theta_{2}}\right)
\]
can be extended to a continuous function on $\left[0,\pi\right]\times\left[0,1\right]\times\left[0,\pi\right]$.
\end{lem}
\begin{proof}
Taking the limit $r\rightarrow1^{-}$ is equivalent to taking the
limit $t\rightarrow\infty$. 
\begin{align*}
h\left(k_{\theta_{1}}a_{t}k_{\theta_{2}}\right) & =f\left(a\frac{\rho\left(k_{\theta_{1}}\right)\rho\left(a_{t}\right)\rho\left(k_{\theta_{2}}\right)}{\norm{\rho\left(k_{\theta_{1}}\right)\rho\left(a_{t}\right)\rho\left(k_{\theta_{2}}\right)}}\right)=f\left(a\rho\left(k_{\theta_{1}}\right)\frac{\rho\left(a_{t}\right)}{\norm{\rho\left(a_{t}\right)}}\rho\left(k_{\theta_{2}}\right)\right).
\end{align*}
Write the decomposition of $V$ as 
\[
V=\bigoplus_{i}V_{k_{i}}
\]
with $V_{k_{i}}$ irreducible of weight $k_{i}$, and such that all
$V_{k_{i}}$ with $k_{i}=k$ ($k=\max k_{i}$) appear in the first
$m$ summands. Recall from \ref{sec:K-invariant-norms} that
\[
\rho\left(a_{t}\right)=\begin{pmatrix}\rho_{k_{1}}\left(a_{t}\right) & 0 & 0\\
0 & \ddots & 0\\
0 & 0 & \rho_{k_{n}}\left(a_{t}\right)
\end{pmatrix}.
\]
with
\[
\rho_{k_{i}}\left(a_{t}\right)=\begin{pmatrix}e^{\frac{k_{i}t}{2}} & 0 & 0 & 0\\
0 & e^{\frac{\left(k_{i}-2\right)t}{2}} & 0 & 0\\
0 & 0 & \ddots & 0\\
0 & 0 & 0 & e^{-\frac{k_{i}t}{2}}
\end{pmatrix}.
\]
Since $\norm{\rho\left(a_{t}\right)}=e^{\frac{kt}{2}}$, we see that
\[
\lim_{t\rightarrow\infty}\frac{\rho_{k}\left(a_{t}\right)}{\norm{\rho\left(a_{t}\right)}}=P_{k}.
\]
Therefore,
\[
\lim_{r\rightarrow1^{-}}\tilde{h}\left(\theta_{1},r,\theta_{2}\right)=\lim_{t\rightarrow\infty}h\left(k_{\theta_{1}}a_{t}k_{\theta_{2}}\right)=f\left(\rho\left(k_{\theta_{1}}\right)\cdot aP_{k}\cdot\rho\left(k_{\theta_{2}}\right)\right).
\]
Since the image of $K$ in $\endd\left(V\right)$ is compact and $f$
is continuous, the limit is uniform in $\theta_{1},\theta_{2}$.
\end{proof}
We can therefore extend $\tilde{h}$ to a continuous function on $\left[0,\pi\right]\times\left[0,1\right]\times\left[0,\pi\right]$.
The proof of the lemma shows that the values of $\tilde{h}$ for $r=1$
are given by
\begin{equation}
\tilde{h}\left(k_{1},1,k_{2}\right)=f\left(\rho\left(k_{\theta_{1}}\right)\cdot aP_{k}\cdot\rho\left(k_{\theta_{2}}\right)\right).\label{eq:boundary-value-of-h}
\end{equation}

We are now ready to study the averages $\lambda_{T}\left(f\right)$
in the general case of a continuous function $f$, and give an error
estimate when $f$ satisfies the H\"{o}lder condition.

\subsubsection{The case of continuous functions}

For continuous functions, the convergence result will be obtained
through the following proposition.
\begin{prop}
\label{prop:annuli_continuous}Let $\psi:\sl_{2}\left(\real\right)\rightarrow\c$
be a continuous function such that the function $\tilde{\psi}:\left[0,\pi\right]\times\left[0,1\right)\times\left[0,\pi\right]$
is given by
\[
\tilde{\psi}\left(\theta_{1},r,\theta_{2}\right)=\psi\left(k_{\theta_{1}}a_{t\left(r\right)}k_{\theta_{2}}\right),
\]
extends to a continuous function on $\left[0,\pi\right]\times\left[0,1\right]\times\left[0,\pi\right]$.
Let $\mu$ be the lift of the Patterson-Sullivan to a measure on $K$
as in remark \ref{rem:symmetric}. Then,
\begin{align*}
\lim_{T\rightarrow\infty}\frac{1}{M_{T}}\sum_{\gamma\in\overline{S}_{T}}\psi\left(\gamma\right) & =\frac{1}{2}V_{\Gamma}\int_{\left[0,\pi\right]\times\left[0,\pi\right]}\tilde{\psi}\left(\theta_{1},1,\theta_{2}\right)\,\d\mu\left(\theta_{1}\right)\d\mu\left(\theta_{2}\right).
\end{align*}
\end{prop}
\begin{proof}
By the Stone-Weierstrass theorem, the algebra spanned by functions
of the form
\[
e^{2in\theta_{1}}\varphi\left(r\right)e^{2\pi im\theta_{2}},\quad n,m\in\z,\varphi\in C\left(\left[0,1\right]\right).
\]
is dense in the algebra of all continuous functions on $\left[0,\pi\right]\times\left[0,1\right]\times\left[0,\pi\right]$
with respect to the maximum norm. Therefore, it is enough to prove
the claim for functions $\psi$ of the form
\[
\psi\left(g\right)=e^{2in\theta_{1}\left(g\right)}\xi\left(r\left(g\right)\right)e^{2im\theta_{2}\left(g\right)},
\]
with $n,m\in\z$ and $\xi\in C\left(\left[0,1\right]\right)$. For
functions of this form,
\[
\tilde{\psi}\left(k_{1},1,k_{2}\right)=e^{2in\theta_{1}}\xi\left(1\right)e^{2im\theta_{2}}.
\]
Now,
\begin{align}
\frac{V_{\Gamma}}{2}\int_{\left[0,\pi\right]\times\left[0,\pi\right]}e^{2in\theta_{1}}\xi\left(1\right)e^{2im\theta_{2}}\,\d\mu\left(\theta_{1}\right)\d\mu\left(\theta_{2}\right) & =\frac{V_{\Gamma}\xi\left(1\right)}{2}\int_{0}^{\pi}e^{2in\theta_{1}}\,\d\mu\left(\theta_{1}\right)\int_{0}^{\pi}e^{2im\theta_{2}}\d\mu\left(\theta_{2}\right)\nonumber \\
 & =\frac{1}{2}V_{\Gamma}\xi\left(1\right)\hat{\mu}\left(2n\right)\hat{\mu}\left(2m\right)\label{eq:Integral-to-Fourier-coeff}
\end{align}
with $\hat{\mu}$ defines as in \cite{bourgain2010sector} (see \ref{sec:Sector-estimates}).
For $\gamma\in S_{T}$,
\[
e^{\frac{kt\left(\gamma\right)}{2}}=\norm{\rho\left(\gamma\right)}\ge T
\]
Therefore, as $T\rightarrow\infty$, 
\[
r\left(\gamma\right)\equiv r\left(t\left(\gamma\right)\right)\rightarrow1
\]
 uniformly over all $\gamma\in S_{T}$. Since $\xi$ is continuous,
$\xi\left(r\left(\gamma\right)\right)\rightarrow\xi\left(1\right)$
uniformly on $S_{T}$ as $T\rightarrow\infty$. Therefore,
\begin{align*}
\lim_{T\rightarrow\infty}\frac{1}{M_{T}}\sum_{\gamma\in\overline{S}_{T}}\psi\left(\gamma\right) & =\lim_{T\rightarrow\infty}\frac{1}{M_{T}}\sum_{\gamma\in\overline{S}_{T}}e^{2in\theta_{1}\left(g\right)}\xi\left(1\right)e^{2im\theta_{2}}.
\end{align*}
Put $\left|\gamma\right|=e^{\frac{1}{2}\d\left(0,\gamma.0\right)}$
as before. By theorem \ref{thm:BKS},
\begin{equation}
\lim_{T\rightarrow\infty}\frac{1}{T^{2\delta}}\sum_{\left|\gamma\right|\le T}e^{2in\theta_{1}\left(\gamma\right)}e^{2im\theta_{2}\left(\gamma\right)}=\frac{V_{\Gamma}}{2}\hat{\mu}\left(2n\right)\hat{\mu}\left(2m\right).\label{eq:(B-K-S)-internal-reference}
\end{equation}
By the definition of $S_{T}$, we have $\overline{S}_{T}=\left\{ \gamma\in\overline{\Gamma}|T<\left|\gamma\right|^{k}\le\left(1+\beta\right)T\right\} .$Hence,
\[
\sum_{\gamma\in\overline{S}_{T}}\psi\left(\gamma\right)=\sum_{\left|\gamma\right|\le\left(\left(1+\beta\right)T\right)^{\frac{1}{k}}}e^{2in\theta_{1}\left(\gamma\right)}\xi\left(1\right)e^{2im\theta_{2}\left(\gamma\right)}-\sum_{\left|\gamma\right|\le T^{\frac{1}{k}}}e^{2in\theta_{1}\left(\gamma\right)}\xi\left(1\right)e^{2im\theta_{2}\left(\gamma\right)}.
\]
Therefore, by \ref{eq:(B-K-S)-internal-reference},
\begin{align*}
\lim_{T\rightarrow\infty}\frac{1}{M_{T}}\sum_{\gamma\in\overline{S}_{T}}\psi\left(\gamma\right) & =\lim_{T\rightarrow\infty}\frac{1}{T^{\frac{2\delta}{k}}\left(\left(1+\beta\right)^{\frac{2\delta}{k}}-1\right)}\sum_{\gamma\in\overline{S}_{T}}\psi\left(\gamma\right)\\
 & =\lim_{T\rightarrow\infty}\frac{1}{T^{\frac{2\delta}{k}}\left(\left(1+\beta\right)^{\frac{2\delta}{k}}-1\right)}\sum_{\left|\gamma\right|\le\left(\left(1+\beta\right)T\right)^{\frac{1}{k}}}e^{2in\theta_{1}\left(\gamma\right)}\xi\left(1\right)e^{2im\theta_{2}\left(\gamma\right)}\\
 & -\lim_{T\rightarrow\infty}\frac{1}{T^{\frac{2\delta}{k}}\left(\left(1+\beta\right)^{\frac{2\delta}{k}}-1\right)}\sum_{\left|\gamma\right|\le T^{\frac{1}{k}}}e^{2in\theta_{1}\left(\gamma\right)}\xi\left(1\right)e^{2im\theta_{2}\left(\gamma\right)}\\
 & =\frac{V_{\Gamma}\xi\left(1\right)}{2}\hat{\mu}\left(2n\right)\hat{\mu}\left(2m\right).
\end{align*}
Thus, by \ref{eq:Integral-to-Fourier-coeff},
\begin{align*}
\lim_{T\rightarrow\infty}\frac{1}{M_{T}}\sum_{\gamma\in\overline{S}_{T}}\psi\left(\gamma\right) & =\frac{V_{\Gamma}}{2}\int_{\left[0,\pi\right]\times\left[0,\pi\right]}e^{2in\theta_{1}}\xi\left(1\right)e^{2im\theta_{2}}\,\d\mu\left(\theta_{1}\right)\d\mu\left(\theta_{2}\right)\\
 & =\frac{V_{\Gamma}}{2}\int_{\left[0,\pi\right]\times\left[0,\pi\right]}\tilde{\psi}\left(\theta_{1},1,\theta_{2}\right)\,\d\mu\left(\theta_{1}\right)\d\mu\left(\theta_{2}\right).
\end{align*}
\end{proof}
\begin{cor}
\label{cor:convergence-h}For $h$ defined in \ref{eq:definition-of-h},
\[
\lim_{T\rightarrow\infty}\frac{1}{M_{T}}\sum_{\gamma\in\overline{S}_{T}}h\left(\gamma\right)=\frac{1}{8}V_{\Gamma}\int_{\left[0,2\pi\right]\times\left[0,2\pi\right]}f\left(\rho\left(k_{\theta_{1}}\right)\cdot aP_{k}\cdot\rho\left(k_{\theta_{2}}\right)\right)\,\d\mu\left(\theta_{1}\right)\d\mu\left(\theta_{2}\right).
\]
\end{cor}
\begin{proof}
Lemma \ref{lem:extension-of-h} shows that $\tilde{h}$ can be extended
to a continuous function on $\left[0,\pi\right]\times\left[0,1\right]\times\left[0,\pi\right]$.
This means that $h$ satisfies the conditions of proposition \ref{prop:annuli_continuous}.
Consequently,
\begin{align*}
\lim_{T\rightarrow\infty}\frac{1}{M_{T}}\sum_{\gamma\in\overline{S}_{T}}h\left(\gamma\right) & =\frac{1}{2}V_{\Gamma}\int_{\left[0,\pi\right]\times\left[0,\pi\right]}\tilde{h}\left(\theta_{1},1,\theta_{2}\right)\,\d\mu\left(\theta_{1}\right)\d\mu\left(\theta_{2}\right).
\end{align*}
From \ref{eq:boundary-value-of-h} we see that the values of $\tilde{h}$
for $r=1$ are given by 
\[
\tilde{h}\left(\theta_{1},1,\theta_{2}\right)=f\left(\rho\left(k_{\theta_{1}}\right)\cdot aP_{k}\cdot\rho\left(k_{\theta_{2}}\right)\right).
\]
Therefore, since $f$ is even under $\rho\left(-1\right)$,
\begin{align*}
\lim_{T\rightarrow\infty}\frac{1}{M_{T}}\sum_{\gamma\in\overline{S}_{T}}h\left(\gamma\right) & =\frac{1}{2}V_{\Gamma}\int_{\left[0,\pi\right]\times\left[0,\pi\right]}f\left(\rho\left(k_{\theta_{1}}\right)\cdot aP_{k}\cdot\rho\left(k_{\theta_{2}}\right)\right)\,\d\mu\left(\theta_{1}\right)\d\mu\left(\theta_{2}\right)\\
 & =\frac{1}{8}V_{\Gamma}\int_{\left[0,2\pi\right]\times\left[0,2\pi\right]}f\left(\rho\left(k_{\theta_{1}}\right)\cdot aP_{k}\cdot\rho\left(k_{\theta_{2}}\right)\right)\,\d\mu\left(\theta_{1}\right)\d\mu\left(\theta_{2}\right).
\end{align*}
\end{proof}
\begin{cor}
\label{cor:Annuli}For $j=1,\dots,N-1$,
\[
\lim_{T\rightarrow\infty}\nu_{T,N,j}\left(f\right)=\frac{1}{4}V_{\Gamma}\int_{K\times K}f\left(\rho\left(k_{1}\right)\cdot\frac{j}{N}P_{k}\cdot\rho\left(k_{2}\right)\right)\,\d\mu\left(k_{1}\right)\d\mu\left(k_{2}\right).
\]
\end{cor}
\begin{proof}
Recall from \ref{eq:alpha-h} that
\[
\frac{2}{M_{T}}\sum_{\gamma\in\overline{S}_{T}}h\left(\gamma\right)=\frac{1}{M_{T}}\sum_{\gamma\in S_{T}}f\left(a\cdot\frac{\rho\left(\gamma\right)}{\norm{\rho\left(\gamma\right)}}\right).
\]
By corollary \ref{cor:convergence-h},
\[
\lim_{T\rightarrow\infty}\frac{1}{M_{T}}\sum_{\gamma\in S_{T}}f\left(a\cdot\frac{\rho\left(\gamma\right)}{\norm{\rho\left(\gamma\right)}}\right)=\frac{1}{4}V_{\Gamma}\int_{K\times K}f\left(\rho\left(k_{1}\right)\cdot aP_{k}\cdot\rho\left(k_{2}\right)\right)\,\d\mu\left(k_{1}\right)\d\mu\left(k_{2}\right).
\]
By setting $T=\frac{T'j}{N}$, $\beta=j^{-1}$, and $a=\frac{j}{N}$
we get the desired result.
\end{proof}

\subsubsection{The case of H\"{o}lder continuous functions\label{subsec:Annuli-Holderof}}

In this subsection, we assume that $f$ is H\"{o}lder continuous with
exponent $\alpha$ and H\"{o}lder constant $\norm f_{\mathrm{Lip}\alpha}$.
The quantitative estimate for the convergence of the annuli averages
$\nu_{T,N,j}\left(f\right)$ will be obtained using the following
proposition.
\begin{prop}
\label{prop:annuli_Holder}Let $h$ be as in \ref{eq:definition-of-h}.
There exists some $T_{0}>0$ such that for every $R\in\nat$, and
$T>T_{0}$ we have:
\begin{align*}
\biggl|\frac{1}{M_{T}}\sum_{\gamma\in\overline{S}_{T}}h\left(\gamma\right) & -\frac{1}{8}V_{\Gamma}\int_{K\times K}f\left(\rho\left(k_{1}\right)\cdot aP_{k}\cdot\rho\left(k_{2}\right)\right)\,\d\mu\left(k_{1}\right)\d\mu\left(k_{2}\right)\biggl|\\
 & \le C_{2}\norm f_{\infty,1}\left(R^{2}T^{\frac{2}{k}\left(s_{1}-\delta\right)}+R^{2}\left(1+2R\right)^{\frac{3}{4}}T^{\frac{1}{4k}\left(1-2\delta\right)}\log\left(T\right)^{\frac{1}{4}}\right)\\
 & +C_{2}\norm f_{\mathrm{Lip}\alpha}\left(R^{-\alpha}\log\left(R\right)+T^{-\alpha}\right)
\end{align*}
for some constant $C_{2}>0$ which does not depend on $f$.
\end{prop}
\begin{rem}
The parameter $R$ in the proposition will measure the degree of the
trigonometric approximation of $h$. After computing the overall error
in the proof of the main theorem, $R$ will be chosen (depending on
$T$) in an optimal way to make the error small.

Before we turn to the proof of proposition \ref{prop:annuli_Holder},
we draw the conclusion for the convergence of the averages $\nu_{T,N,j}\left(f\right)$.
\end{rem}
\begin{cor}
\label{cor:annuli-holder-N} There exists a $T_{0}>0$ and $C_{2}>0$
which does not depend on $f$, such that for $T>T_{0}N$ and for all
$R\in\nat$,
\begin{align*}
\biggl|\nu_{T,N,j}\left(f\right) & -\frac{1}{4}V_{\Gamma}\int_{K\times K}f\left(\rho\left(k_{1}\right)\cdot\frac{j}{N}P_{k}\cdot\rho\left(k_{2}\right)\right)\,\d\mu\left(k_{1}\right)\d\mu\left(k_{2}\right)\biggl|\\
 & \le C_{2}\norm f_{\infty,1}\left(R^{2}\left(\frac{Tj}{N}\right)^{\frac{2}{k}\left(s_{1}-\delta\right)}+R^{2}\left(1+2R\right)^{\frac{3}{4}}\left(\frac{Tj}{N}\right)^{\frac{1}{4k}\left(1-2\delta\right)}\log\left(T\right)^{\frac{1}{4}}\right)\\
 & +C_{2}\norm f_{\mathrm{Lip}\alpha}\left(R^{-\alpha}\log\left(R\right)+\left(\frac{Tj}{N}\right)^{-\alpha}\right).
\end{align*}
for all $1\le j\le N$.
\end{cor}
\begin{proof}
By applying \ref{prop:annuli_Holder} to the values $a=\frac{j}{N}$,
$\beta=\frac{1}{j}$ and $T'=\frac{Tj}{N}$ we get the desired inequality
for any $T$ such that $\frac{Tj}{N}>T_{0}$. Since $j\ge1$, it is
enough to take $T>NT_{0}$.
\end{proof}
Our goal in the rest of this section is to prove proposition \ref{prop:annuli_Holder}.
First, we can utilize the H\"{o}lder condition in order to let the radial
part of $h\left(\gamma\right)$ tend to infinity. Indeed, for any
$t\ge0$, the largest possible eigenvalue of $\rho\left(a_{t}\right)$
which is smaller than $e^{\frac{kt}{2}}$ is $e^{\frac{\left(k-1\right)t}{2}}$
. Since $\rho\left(a_{t}\right)$ and $P_{k}$ are both diagonal in
some orthonormal basis,
\begin{equation}
\norm{\frac{\rho_{k}\left(a_{t}\right)}{\norm{\rho\left(a_{t}\right)}}-P_{k}}\le e^{-\frac{t}{2}}.\label{eq:error-boundary}
\end{equation}
Note that this is the case only if $V_{k-1}$ also appears in $V$.
Otherwise, the second largest entry is $e^{\frac{\left(k-2\right)t}{2}}$.
This means that for $\gamma\in S_{T}$ , with $\gamma=k_{1}\rho\left(a_{t}\right)k_{2}$,
\[
\norm{a\frac{\rho\left(\gamma\right)}{\norm{\rho\left(\gamma\right)}}-\rho\left(k_{1}\right)\cdot aP_{k}\cdot\rho\left(k_{2}\right)}=a\cdot\norm{\frac{\rho\left(a_{t}\right)}{\norm{\rho\left(a_{t}\right)}}-P_{k}}\le ae^{-\frac{t}{2}}.
\]
Hence, by the H\"{o}lder property of $f$, for $t>0$ such that $e^{\frac{t}{2}}\ge T$,
we have
\[
\left|h\left(\gamma\right)-f\left(\rho\left(k_{1}\right)\cdot aP_{k}\cdot\rho\left(k_{2}\right)\right)\right|\le\norm f_{\mathrm{Lip}\alpha}T^{-\alpha}.
\]
Recall that we denoted the image of $S_{T}$ in $\psl_{2}\left(\real\right)$
by $\overline{S}_{T}$. Then, for any $T>0$,
\begin{equation}
\left|\frac{1}{M_{T}}\sum_{\gamma\in\overline{S}_{T}}h\left(\gamma\right)-\frac{1}{M_{T}}\sum_{\gamma\in\overline{S}_{T}}f\left(\rho\left(k_{1}\right)\cdot aP_{k}\cdot\rho\left(k_{2}\right)\right)\right|\le\frac{\left|\overline{S}_{T}\right|}{M_{T}}\norm f_{\mathrm{Lip}\alpha}T^{-\alpha}\le C_{\Gamma}\norm f_{\mathrm{Lip}\alpha}T^{-\alpha}.\label{eq:f-to-h-error}
\end{equation}
Define a function $\varphi:\left[0,\pi\right]\times\left[0,\pi\right]\rightarrow\c$
by
\[
\varphi\left(\theta_{1},\theta_{2}\right)=f\left(\rho\left(k_{\theta_{1}}\right)\cdot aP_{k}\cdot\rho\left(k_{\theta_{2}}\right)\right).
\]
We extend $\varphi$ to a function on $\left[0,2\pi\right]\times\left[0,2\pi\right]$
by setting
\[
\varphi\left(\theta_{1}+\pi,\theta_{2}\right)=\varphi\left(\theta_{1},\theta_{2}+\pi\right)=\varphi\left(\theta_{1},\theta_{2}\right).
\]
We denote the extended function by $\varphi$ as well. On $\left[0,2\pi\right]\times\left[0,2\pi\right]\cong\mathbb{T}^{2}$
we fix the metric
\[
\d\left(\left(\theta_{1},\theta_{2}\right),\left(\phi_{1},\phi_{2}\right)\right)=\max\left\{ \left|\theta_{1}-\phi_{1}\right|,\left|\theta_{2}-\phi_{2}\right|\right\} ,\quad\theta_{i},\phi_{i}\in\left[0,2\pi\right].
\]

\begin{lem}
$\varphi:\mathbb{T}^{2}\rightarrow\c$ is H\"{o}lder continuous with exponent
$\alpha$ and constant $b\cdot\norm f_{\mathrm{Lip}\alpha}$, where
$b$ is some constant which depends only on $\rho$ and the norm $\norm .$
on $V$.
\end{lem}
\begin{proof}
Consider the function $q:\left[0,2\pi\right]\times\left[0,2\pi\right]\rightarrow\endd\left(V\right)$
given by
\[
q\left(\theta_{1},\theta_{2}\right)=\rho\left(k\left(\theta_{1}\right)\right)\cdot aP_{k}\cdot\rho\left(k\left(\theta_{2}\right)\right),\quad\theta_{1},\theta_{2}\in\left[0,2\pi\right].
\]
Then, $\varphi=f\circ q$. As we have seen in \ref{sec:Basic-structure-and-rep=00003DSL2},
matrix elements of dimensional representations of $\sl_{2}\left(\real\right)$
are polynomials in the entries of the elements in $\sl_{2}\left(\real\right)$
of degree smaller than or equal to $\dim V$. Therefore, the elements
of the matrix $q\left(\theta_{1},\theta_{2}\right)\in\endd\left(V\right)$
are polynomials in $\cos\theta_{i},\sin\theta_{i}$, $i=1,2$ of degree
at most $\dim\left(V\right)$. The derivative of $q\left(\theta_{1},\theta_{2}\right)$
is then bounded by some constant $b$ which depends on $\rho$ and
the norm we have fixed. Consequently, $q$ is Lipschitz continuous
with Lipschitz constant $b$. Given that $f$ is $\left(\norm f_{\mathrm{Lip}\alpha},\alpha\right)$-
H\"{o}lder, $\varphi$ is H\"{o}lder continuous with exponent $\alpha$ and
constant $b\norm f_{\mathrm{Lip}\alpha}$ as the composition $f\circ q$. 
\end{proof}
We now turn to the proof of proposition \ref{prop:annuli_Holder}.
As we have seen in \ref{eq:f-to-h-error},
\begin{equation}
\left|\frac{1}{M_{T}}\sum_{\gamma\in\overline{S}_{T}}h\left(\gamma\right)-\frac{1}{M_{T}}\sum_{\gamma\in\overline{S}_{T}}\varphi\left(\theta_{1}\left(\gamma\right),\theta_{2}\left(\gamma\right)\right)\right|\le C_{\Gamma}\norm f_{\mathrm{Lip}\alpha}T^{-\alpha}.\label{eq:h-to-phi-error}
\end{equation}
We will study the averages $\frac{1}{M_{T}}\sum_{\gamma\in\overline{S}_{T}}\varphi\left(\theta_{1}\left(\gamma\right),\theta_{2}\left(\gamma\right)\right)$
by using the quantitative result in theorem \ref{thm:BKS}. In order
to do that, we need to quantify the approximation of $\varphi$ by
trigonometric polynomials. 

We will approximate $\varphi$ by its Fej\'{e}r means. The $R$-th Fej\'{e}r
mean of $\varphi$ will be denoted by $\sigma_{R}$. It is defined
by
\[
\sigma_{R}\left(\theta_{1},\theta_{2}\right)=\frac{1}{4\pi^{2}}\int_{\mathbb{T}^{2}}\varphi\left(\theta_{1}-u,\theta_{2}-v\right)F_{R}\left(u,v\right)\,\d u\d v,
\]
Where,
\[
F_{R}\left(u,v\right)=F_{R}\left(u\right)F_{R}\left(v\right)=\frac{1}{R}\left(\frac{\sin\left(\frac{Ru}{2}\right)}{\sin\left(\frac{u}{2}\right)}\right)^{2}\frac{1}{R}\left(\frac{\sin\left(\frac{Rv}{2}\right)}{\sin\left(\frac{v}{2}\right)}\right)^{2}
\]
is a product of one dimensional Fej\'{e}r kernels. $\sigma_{R}$ is a
trigonometric polynomial of degree $R$. It is of the form
\[
\sigma_{R}\left(\theta_{1},\theta_{2}\right)=\sum_{n,m=-R}^{R}\varphi_{n,m}e^{2ni\theta_{1}+2mi\theta_{2}},
\]
with $\left|\varphi_{n,m}\right|\le\sup_{\theta_{1},\theta_{2}}\left|\varphi\left(\theta_{1},\theta_{2}\right)\right|$.
Only even frequencies appear in $\sigma_{R}$ since $\varphi$ is
symmetric under $\theta\mapsto\theta+\pi$ in both coordinates. Since
$\varphi$ is H\"{o}lder continuous with exponent $\alpha$ and constant
$b\cdot\norm f_{\mathrm{Lip}\alpha}$, there is some constant $C_{0}>0$
such that for all $R\in\nat$ and for all $\theta_{1},\theta_{2}\in\left[0,2\pi\right]$,
\begin{align}
\left|\varphi\left(\theta_{1},\theta_{2}\right)-\sigma_{R}\left(\theta_{1},\theta_{2}\right)\right|\le & C_{0}b\norm f_{\mathrm{Lip}\alpha}R^{-\alpha}\log R.\label{eq:phi-to-theta}
\end{align}
For the proof, see Appendix \ref{sec:Appendix---Convergence}. We
have,
\begin{align*}
\frac{1}{M_{T}}\sum_{\gamma\in\overline{S}_{T}}\sigma_{R}\left(\theta_{1}\left(\gamma\right),\theta_{2}\left(\gamma\right)\right) & =\frac{1}{M_{T}}\sum_{\gamma\in\overline{S}_{T}}\sum_{n,m=-R}^{R}\varphi_{n,m}e^{2ni\theta_{1}\left(\gamma\right)+2mi\theta_{2}\left(\gamma\right)}\\
 & =\sum_{n,m=-R}^{R}\varphi_{n,m}\frac{1}{M_{T}}\sum_{\gamma\in\overline{S}_{T}}e^{2ni\theta_{1}\left(\gamma\right)+2mi\theta_{2}\left(\gamma\right)}.
\end{align*}
Now, by the definition of $\overline{S}_{T}$,
\[
\frac{1}{M_{T}}\sum_{\gamma\in\overline{S}_{T}}e^{2ni\theta_{1}\left(\gamma\right)+2mi\theta_{2}\left(\gamma\right)}=\frac{1}{T^{\frac{2\delta}{k}}\left(\left(1+\beta\right)^{\frac{2\delta}{k}}-1\right)}\sum_{T^{\frac{1}{k}}\le\left|\gamma\right|\le\left(\left(1+\beta\right)T\right)^{\frac{1}{k}}}e^{2ni\theta_{1}\left(\gamma\right)+2mi\theta_{2}\left(\gamma\right)}.
\]
By theorem \ref{thm:BKS}, there exists a $\tilde{T}_{0}>0$ and a
constant $\tilde{C}_{3}$ which does not depend on $n,m$ such that
for $T'>\tilde{T}_{0}$,
\begin{align*}
\left|\frac{1}{\left(T'\right)^{2\delta}}\sum_{\left|\gamma\right|\le T'}e^{i\left(2n\theta_{1}\left(\gamma\right)+2m\theta_{2}\left(\gamma\right)\right)}-\frac{1}{2}V_{\Gamma}\hat{\mu}\left(2n\right)\hat{\mu}\left(2m\right)\right| & \le\tilde{C}_{3}\left(\left(T'\right)^{2\left(s_{1}-\delta\right)}+\left(T'\right)^{\frac{1}{4}\left(1-2\delta\right)}\log\left(T'\right)^{\frac{1}{4}}\left(1+\left|n\right|+\left|m\right|\right)^{\frac{3}{4}}\right).
\end{align*}
Therefore, taking $T'=T^{\frac{1}{k}}$, and setting $T_{0}=\left(\tilde{T}_{0}\right)^{\frac{1}{k}}$,
we get that there is a constant $C_{3}>0$ such that for $T>T_{0}$,
\begin{align}
\left|\frac{1}{M_{T}}\sum_{\gamma\in\overline{S}_{T}}e^{i\left(2n\theta_{1}\left(\gamma\right)+2m\theta_{2}\left(\gamma\right)\right)}-\frac{1}{2}V_{\Gamma}\hat{\mu}\left(2n\right)\hat{\mu}\left(2m\right)\right| & \leq C_{3}\left(T^{\frac{2}{k}\left(s_{1}-\delta\right)}+T^{\frac{1}{4k}\left(1-2\delta\right)}\log\left(T\right)^{\frac{1}{4}}\left(1+\left|n\right|+\left|m\right|\right)^{\frac{3}{4}}\right)\label{eq:BKS-annuli-quant-bound}
\end{align}
From \ref{eq:BKS-annuli-quant-bound} we see that we can bound the
difference
\begin{equation}
\left|\frac{1}{M_{T}}\sum_{\gamma\in\overline{S}_{T}}\sigma_{R}\left(\theta_{1}\left(\gamma\right),\theta_{2}\left(\gamma\right)\right)-\frac{1}{2}\sum_{n,m=-R}^{R}\varphi_{n,m}V_{\Gamma}\hat{\mu}\left(2n\right)\hat{\mu}\left(2m\right)\right|\label{eq:(difference-annuli)}
\end{equation}
by the sum
\begin{equation}
\sum_{n,m=-R}^{R}\varphi_{n,m}C_{3}\left(T^{\frac{2}{k}\left(s_{1}-\delta\right)}+T^{\frac{1}{4k}\left(1-2\delta\right)}\log\left(T\right)^{\frac{1}{4}}\left(1+\left|n\right|+\left|m\right|\right)^{\frac{3}{4}}\right)\label{eq:difference-annuli-bound-1}
\end{equation}
Since $\varphi_{n,m}$ is bounded by the maximum of $\varphi$ on
$\left[0,2\pi\right]\times\left[0,2\pi\right]$ for any $n,m\in\z$,
it is bounded by $\norm f_{\infty,1}$. Also,
\[
\left(1+\left|n\right|+\left|m\right|\right)^{\frac{3}{4}}\le\left(1+2R\right)^{\frac{3}{4}}.
\]
Therefore, \ref{eq:difference-annuli-bound-1} is bounded by
\[
4R^{2}C_{3}\norm f_{\infty,1}\left(T^{\frac{2}{k}\left(s_{1}-\delta\right)}+\left(1+2R\right)^{\frac{3}{4}}T^{\frac{1}{4k}\left(1-2\delta\right)}\log\left(T\right)^{\frac{1}{4}}\right).
\]
On the other hand, we have:
\begin{align*}
\int_{K\times K}\sigma_{R}\left(\theta_{1},\theta_{2}\right)\,\d\mu\left(\theta_{1}\right)\d\mu\left(\theta_{2}\right) & =\int_{\mathbb{T}^{2}}\sigma_{R}\left(\theta_{1},\theta_{2}\right)\,\d\mu\left(\theta_{1}\right)\d\mu\left(\theta_{2}\right)\\
 & =\int_{\mathbb{T}^{2}}\sum_{n,m=-R}^{R}\varphi_{n,m}e^{2in\theta_{1}+2im\theta_{2}}\,\d\mu\left(\theta_{1}\right)\d\mu\left(\theta_{2}\right)\\
 & =\sum_{n,m=-R}^{R}\varphi_{n,m}\int_{\mathbb{T}^{2}}e^{2in\theta_{1}+2im\theta_{2}}\,\d\mu\left(\theta_{1}\right)\d\mu\left(\theta_{2}\right)\\
 & =\sum_{n,m=-R}^{R}\varphi_{n,m}\left(\int_{0}^{2\pi}e^{2in\theta_{1}}\,\d\mu\left(\theta_{1}\right)\right)\left(\int_{0}^{2\pi}e^{2im\theta_{2}}\,\d\mu\left(\theta_{2}\right)\right).
\end{align*}
Recall that by the definition of the Fourier coefficients of the measure
$\mu$ in \ref{sec:Sector-estimates},
\[
\int_{0}^{2\pi}e^{2in\theta_{1}}\,\d\mu\left(\theta_{1}\right)=2\int_{0}^{\pi}e^{2in\theta_{1}}\,\d\mu\left(\theta_{1}\right)=2\hat{\mu}\left(2n\right).
\]
Then,
\[
\int_{K\times K}\sigma_{R}\left(\theta_{1},\theta_{2}\right)\,\d\mu\left(\theta_{1}\right)\d\mu\left(\theta_{2}\right)=4\sum_{n,m=-R}^{R}\varphi_{n,m}\hat{\mu}\left(2n\right)\hat{\mu}\left(2m\right).
\]
This gives
\[
\frac{1}{8}V_{\Gamma}\int_{K\times K}\sigma_{R}\left(\theta_{1},\theta_{2}\right)\,\d\mu\left(\theta_{1}\right)\d\mu\left(\theta_{2}\right)=\frac{1}{2}\sum_{n,m=-R}^{R}\varphi_{n,m}V_{\Gamma}\hat{\mu}\left(2n\right)\hat{\mu}\left(2m\right).
\]
Recall that for any $\theta_{1},\theta_{2}\in\left[0,2\pi\right]$,
\[
\left|\sigma_{R}\left(\theta_{1},\theta_{2}\right)-\varphi\left(\theta_{1},\theta_{2}\right)\right|\le C_{0}b\norm f_{\mathrm{Lip}\alpha}R^{-\alpha}\log R,
\]
and that, by definition $\varphi\left(\theta_{1},\theta_{2}\right)=f\left(\rho\left(k_{\theta_{1}}\right)\cdot aP_{k}\cdot\rho\left(k_{\theta_{2}}\right)\right).$
Therefore, the difference
\[
\left|\frac{1}{M_{T}}\sum_{\gamma\in\overline{S}_{T}}\varphi\left(\theta_{1}\left(\gamma\right),\theta_{2}\left(\gamma\right)\right)-\frac{1}{8}V_{\Gamma}\int_{K\times K}f\left(\rho\left(k_{1}\right)\cdot aP_{k}\cdot\rho\left(k_{2}\right)\right)\,\d\mu\left(k_{1}\right)\d\mu\left(k_{2}\right)\right|
\]
is bounded by,
\[
\left(C_{0}b+C_{\Gamma}\right)\norm f_{\mathrm{Lip}\alpha}R^{-\alpha}\log R+4R^{2}C_{3}\norm f_{\infty,1}\left(T^{\frac{2}{k}\left(s_{1}-\delta\right)}+\left(1+2R\right)^{\frac{3}{4}}T^{\frac{1}{4k}\left(1-2\delta\right)}\log\left(T\right)^{\frac{1}{4}}\right).
\]
In total, for $T>T_{0}$ 
\begin{align*}
\biggl|\frac{1}{M_{T}}\sum_{\gamma\in\overline{S}_{T}}h\left(\gamma\right) & -\frac{1}{8}V_{\Gamma}\int_{K\times K}f\left(\rho\left(k_{1}\right)\cdot aP_{k}\cdot\rho\left(k_{2}\right)\right)\,\d\mu\left(k_{1}\right)\d\mu\left(k_{2}\right)\biggl|\\
 & \le\left(C_{0}b+C_{\Gamma}\right)\norm f_{\mathrm{Lip}\alpha}R^{-\alpha}\log R\\
 & +4R^{2}C_{3}\norm f_{\infty,1}\left(T^{\frac{2}{k}\left(s_{1}-\delta\right)}+\left(1+2R\right)^{\frac{3}{4}}T^{\frac{1}{4k}\left(1-2\delta\right)}\log\left(T\right)^{\frac{1}{4}}\right)\\
 & +C_{\Gamma}\norm f_{\mathrm{Lip}\alpha}T^{-\alpha}.
\end{align*}
This ends the proof of proposition \ref{prop:annuli_Holder}.

\subsection{The conclusion of the main theorems}

\subsubsection{The main theorem for continuous functions}

We now turn to the proof of theorem \ref{thm:normalized}. We wish
to estimate the difference between $\nu_{T}\left(f\right)$ and $\nu_{\Gamma}\left(f\right)$.
As we have seen, $\nu_{T}\left(f\right)$ can be approximated by the
averages $\nu_{T,N}\left(f\right)$. The first step is to establish
the convergence of $\nu_{T,N}\left(f\right)$ as $T\rightarrow\infty$. 
\begin{lem}
\label{lem:T_infinity_1}For all $N\in\nat$,
\[
\lim_{T\rightarrow\infty}\nu_{T,N}\left(f\right)=\frac{\delta}{2k}V_{\Gamma}\sum_{j=1}^{N-1}\frac{1}{N}\left(\frac{j}{N}\right)^{\frac{2\delta}{k}-1}\int_{K\times K}f\left(\rho\left(k_{1}\right)\cdot\frac{j}{N}P_{k}\cdot\rho\left(k_{2}\right)\right)\,\d\mu\left(k_{1}\right)\d\mu\left(k_{2}\right).
\]
\end{lem}
\begin{proof}
Recall that
\[
\nu_{T,N}\left(f\right)=\frac{2\delta}{k}\sum_{j=1}^{N-1}\frac{1}{N}\left(\frac{j}{N}\right)^{\frac{2\delta}{k}-1}\nu_{T,N,j}\left(f\right).
\]
Then,
\[
\lim_{T\rightarrow\infty}\nu_{T,N}\left(f\right)=\frac{2\delta}{k}\sum_{j=1}^{N-1}\frac{1}{N}\left(\frac{j}{N}\right)^{\frac{2\delta}{k}-1}\lim_{T\rightarrow\infty}\nu_{T,N,j}\left(f\right).
\]
By corollary \ref{cor:Annuli},
\[
\lim_{T\rightarrow\infty}\nu_{T,N,j}\left(f\right)=\frac{1}{4}V_{\Gamma}\int_{K\times K}f\left(\rho\left(k_{1}\right)\cdot\frac{j}{N}P_{k}\cdot\rho\left(k_{2}\right)\right)\,\d\mu\left(k_{1}\right)\d\mu\left(k_{2}\right).
\]
Therefore, the limit of $\nu_{T,N}$ is given by
\begin{align*}
\lim_{T\rightarrow\infty}\nu_{T,N}\left(f\right) & =\frac{\delta}{2k}V_{\Gamma}\sum_{j=1}^{N-1}\frac{1}{N}\left(\frac{j}{N}\right)^{\frac{2\delta}{k}-1}\int_{K\times K}f\left(\rho\left(k_{1}\right)\cdot\frac{j}{N}P_{k}\cdot\rho\left(k_{2}\right)\right)\,\d\mu\left(k_{1}\right)\d\mu\left(k_{2}\right).
\end{align*}
\end{proof}
We wish to study the right hand side of the limit from lemma \ref{lem:T_infinity_1}
as $N\rightarrow\infty$. We can write the expression in the limit
as the integral $\int_{0}^{1}G_{N}\left(t\right)\,\d t$, where,
\[
G_{N}\left(t\right)=\frac{\delta}{2k}V_{\Gamma}\sum_{j=1}^{N-1}\chi_{\left[\frac{j}{N},\frac{j+1}{N}\right]}\left(t\right)\left(\frac{j}{N}\right)^{\frac{2\delta}{k}-1}\int_{K\times K}f\left(\rho\left(k_{1}\right)\cdot\frac{j}{N}P_{k}\cdot\rho\left(k_{2}\right)\right)\,\d\mu\left(k_{1}\right)\d\mu\left(k_{2}\right),
\]
where $\chi_{\left[a,b\right]}$ is the indicator function of the
interval $\left[a,b\right]$ for $a,b\in\real$. For every $t\in\left(0,1\right)$,
\[
\lim_{N\rightarrow\infty}G_{N}\left(t\right)=\frac{\delta}{2k}V_{\Gamma}t^{\frac{2\delta}{k}-1}\int_{K\times K}f\left(\rho\left(k_{1}\right)\cdot tP_{k}\cdot\rho\left(k_{2}\right)\right)\,\d\mu\left(k_{1}\right)\d\mu\left(k_{2}\right).
\]
The functions $\left\{ G_{N}\right\} _{N\in\nat}$ are uniformly bounded
by
\[
\left|G_{N}\left(t\right)\right|\le\frac{\delta}{2k}V_{\Gamma}t^{\frac{2\delta}{k}-1}\sup_{\norm x\le1}\left|f\left(x\right)\right|\mu\left(K\right)\mu\left(K\right)=\frac{2\delta}{k}V_{\Gamma}t^{\frac{2\delta}{k}-1}\norm f_{\infty,1}.
\]
The function $t^{\frac{2\delta}{k}-1}$ is integrable on $\left[0,1\right]$.
Hence, by the dominant convergence theorem applied to the sequence
$\left\{ G_{N}\right\} _{N\in\nat}$, we get:
\begin{align}
\lim_{N\rightarrow\infty}\int_{0}^{1}G_{N}\left(t\right)\,\d t & =\frac{\delta}{2k}V_{\Gamma}\int_{0}^{1}t^{\frac{2\delta}{k}-1}\left(\int_{K\times K}f\left(\rho\left(k_{1}\right)\cdot tP_{k}\cdot\rho\left(k_{2}\right)\right)\,\d\mu\left(k_{1}\right)\d\mu\left(k_{2}\right)\right)\,\d t\nonumber \\
 & =\frac{\delta}{2k}V_{\Gamma}\int_{K\times\left[0,1\right]\times K}t^{\frac{2\delta}{k}-1}f\left(\rho\left(k_{1}\right)\cdot tP_{k}\cdot\rho\left(k_{2}\right)\right)\,\d\mu\left(k_{1}\right)\d t\d\mu\left(k_{2}\right).\label{eq:G_N-to-infinity}
\end{align}
The right hand side of \ref{eq:G_N-to-infinity} equals $\int f\,\d\nu_{\Gamma}$.
Thus, by lemma \ref{lem:T_infinity_1}, 
\begin{equation}
\lim_{N\rightarrow\infty}\left(\lim_{T\rightarrow\infty}\nu_{T,N}\left(f\right)\right)=\int f\,\d\nu_{\Gamma}.\label{eq:N-to-infinity-cont}
\end{equation}
To finish the proof of theorem \ref{thm:normalized}, recall that
we have seen in \ref{eq:SUM-TO-ANNULI} that there are $N_{2}>0$
and $C_{0}>0$ such that for $N>N_{2}$ and $T>0$,
\[
\left|\nu_{T}\left(f\right)-\nu_{T,N}\left(f\right)\right|\le C_{0}\left(\epsilon+\epsilon^{\frac{2\delta}{k}}\right).
\]
From \ref{eq:N-to-infinity-cont}, there is some $N_{3}>0$ such that
for $N\ge N_{3}$,
\[
\left|\lim_{T\rightarrow\infty}\nu_{T,N}\left(f\right)-\int f\,\d\nu_{\Gamma}\right|<\epsilon.
\]
In particular, the inequality holds for $N_{4}=\max\left\{ N_{3},N_{2}\right\} $.
Let $T_{0}>0$ be such that for all $T>T_{0}$,
\[
\left|\nu_{T,N_{4}}\left(f\right)-\lim_{T\rightarrow\infty}\nu_{T,N_{4}}\left(f\right)\right|<\epsilon.
\]
Then, for all $T>T_{0}$,
\begin{align*}
\left|\nu_{T}\left(f\right)-\int f\,\d\nu_{\Gamma}\right| & \le\left|\nu_{T}\left(f\right)-\nu_{T,N_{4}}\left(f\right)\right|+\left|\nu_{T,N_{4}}\left(f\right)-\lim_{T\rightarrow\infty}\nu_{T,N_{4}}\left(f\right)\right|+\left|\lim_{T\rightarrow\infty}\nu_{T,N_{4}}\left(f\right)-\int f\,\d\nu_{\Gamma}\right|\\
 & <C_{0}\left(\epsilon+\epsilon^{\frac{2\delta}{k}}\right)+\epsilon+\epsilon\le\left(2+C_{0}\right)\left(\epsilon+\epsilon^{\frac{2\delta}{k}}\right).
\end{align*}
This concludes the proof of theorem \ref{thm:normalized}.

\subsubsection{The quantitative theorem for H\"{o}lder continuous functions}

Recall from \ref{eq:Par-to-N-error} that we have
\begin{equation}
\left|\nu_{T}\left(f\right)-\nu_{T,N}\left(f\right)\right|\le C_{1}\norm f_{\infty,1}\left(N^{-\frac{2\delta}{k}}+N^{-1}\log N\right)+\norm f_{\mathrm{Lip}\alpha}C_{\Gamma}N^{-\alpha}.\label{eq:T-to-TN}
\end{equation}
By corollary \ref{cor:annuli-holder-N}, there exists a $T_{0}>0$
such that for $T>NT_{0}$ and $R\in\nat$,
\begin{align*}
\biggl|\nu_{T,N,j}\left(f\right) & -\frac{1}{4}V_{\Gamma}\int_{K\times K}f\left(\rho\left(k_{1}\right)\cdot\frac{j}{N}P_{k}\cdot\rho\left(k_{2}\right)\right)\,\d\mu\left(k_{1}\right)\d\mu\left(k_{2}\right)\biggl|\\
 & \le C_{2}\norm f_{\infty,1}\left(R^{2}\left(\frac{Tj}{N}\right)^{\frac{2}{k}\left(s_{1}-\delta\right)}+R^{2}\left(1+2R\right)^{\frac{3}{4}}\left(\frac{Tj}{N}\right)^{\frac{1}{4k}\left(1-2\delta\right)}\log\left(T\right)^{\frac{1}{4}}\right)\\
 & +C_{2}\norm f_{\mathrm{Lip}\alpha}\left(R^{-\alpha}\log\left(R\right)+\left(\frac{Tj}{N}\right)^{-\alpha}\right).
\end{align*}
Thus, for $T>NT_{0}$,
\begin{align*}
\biggl|\nu_{T,N}\left(f\right) & -\frac{\delta}{2k}V_{\Gamma}\sum_{j=1}^{N-1}\frac{1}{N}\left(\frac{j}{N}\right)^{\frac{2\delta}{k}-1}\int_{K\times K}f\left(\rho\left(k_{1}\right)\cdot\frac{j}{N}P_{k}\cdot\rho\left(k_{2}\right)\right)\,\d\mu\left(k_{1}\right)\d\mu\left(k_{2}\right)\biggl|\\
 & \le\frac{2\delta}{k}\sum_{j=1}^{N-1}\frac{1}{N}\left(\frac{j}{N}\right)^{\frac{2\delta}{k}-1}\left|\nu_{T,N,j}\left(f\right)-\frac{1}{4}V_{\Gamma}\int_{K\times K}f\left(\rho\left(k_{1}\right)\cdot\frac{j}{N}P_{k}\cdot\rho\left(k_{2}\right)\right)\,\d\mu\left(k_{1}\right)\d\mu\left(k_{2}\right)\right|\\
 & \le\frac{2\delta}{k}\sum_{j=1}^{N-1}\frac{1}{N}\left(\frac{j}{N}\right)^{\frac{2\delta}{k}-1}\cdot C_{2}\norm f_{\infty,1}\left(R^{2}\left(\frac{Tj}{N}\right)^{\frac{2}{k}\left(s_{1}-\delta\right)}+R^{2}\left(1+2R\right)^{\frac{3}{4}}\left(\frac{Tj}{N}\right)^{\frac{1}{4k}\left(1-2\delta\right)}\log\left(T\right)^{\frac{1}{4}}\right)\\
 & +\frac{2\delta}{k}\sum_{j=1}^{N-1}\frac{1}{N}\left(\frac{j}{N}\right)^{\frac{2\delta}{k}-1}C_{2}\norm f_{\mathrm{Lip}\alpha}\left(R^{-\alpha}\log\left(R\right)+\left(\frac{Tj}{N}\right)^{-\alpha}\right).
\end{align*}
Since $f$ is H\"{o}lder continuous and $t^{\frac{2\delta}{k}-1}$ is
integrable on $\left[0,1\right]$, the difference between the sum
\[
\sum_{j=1}^{N-1}\frac{1}{N}\left(\frac{j}{N}\right)^{\frac{2\delta}{k}-1}\int_{K\times K}f\left(\rho\left(k_{1}\right)\cdot\frac{j}{N}P_{k}\cdot\rho\left(k_{2}\right)\right)\,\d\mu\left(k_{1}\right)\d\mu\left(k_{2}\right)
\]
and the integral
\[
\int_{0}^{1}t^{\frac{2\delta}{k}-1}\left(\int_{K\times K}f\left(\rho\left(k_{1}\right)\cdot tP_{k}\cdot\rho\left(k_{2}\right)\right)\,\d\mu\left(k_{1}\right)\d\mu\left(k_{2}\right)\right)\d t
\]
is bounded by a constant $C$ times $\norm f_{\mathrm{Lip}\alpha}N^{-\alpha}$.
Thus, for $T>NT_{0}$,
\begin{align*}
\biggl|\nu_{T,N}\left(f\right) & -\frac{\delta}{2k}V_{\Gamma}\int_{K\times\left[0,1\right]\times K}t^{\frac{2\delta}{k}-1}f\left(\rho\left(k_{1}\right)\cdot tP_{k}\cdot\rho\left(k_{2}\right)\right)\,\d t\d\mu\left(k_{1}\right)\d\mu\left(k_{2}\right)\biggl|\\
 & \le C\norm f_{\mathrm{Lip}\alpha}N^{-\alpha}\\
 & +\frac{2\delta}{k}\sum_{j=1}^{N-1}\frac{1}{N}\left(\frac{j}{N}\right)^{\frac{2\delta}{k}-1}\cdot C_{2}\norm f_{\infty,1}\left(R\cdot\left(\frac{Tj}{N}\right)^{\frac{2}{k}\left(s_{1}-\delta\right)}+R\left(1+R\right)^{\frac{3}{4}}\left(\frac{Tj}{N}\right)^{\frac{1}{4k}\left(1-2\delta\right)}\log\left(T\right)^{\frac{1}{4}}\right)\\
 & +\frac{2\delta}{k}\sum_{j=1}^{N-1}\frac{1}{N}\left(\frac{j}{N}\right)^{\frac{2\delta}{k}-1}C_{2}\norm f_{\mathrm{Lip}\alpha}\left(R^{-\alpha}\log\left(R\right)+\left(\frac{Tj}{N}\right)^{-\alpha}\right).
\end{align*}
Since
\[
\int f\,\d\nu_{\Gamma}=\frac{\delta}{2k}V_{\Gamma}\int_{K\times\left[0,1\right]\times K}t^{\frac{2\delta}{k}-1}f\left(\rho\left(k_{1}\right)\cdot tP_{k}\cdot\rho\left(k_{2}\right)\right)\,\d t\d\mu\left(k_{1}\right)\d\mu\left(k_{2}\right)
\]
the inequality above combined with \ref{eq:T-to-TN} gives, for $T>NT_{0}$
and $R\in\nat$,
\begin{align}
\left|\nu_{T}\left(f\right)-\int f\,\d\nu_{\Gamma}\right| & \le C_{1}\norm f_{\infty,1}\left(N^{-\frac{2\delta}{k}}+N^{-1}\log N\right)+\left(C_{\Gamma}+C\right)\norm f_{\mathrm{Lip}\alpha}N^{-\alpha}\nonumber \\
 & +\frac{2\delta}{k}\sum_{j=1}^{N-1}\frac{1}{N}\left(\frac{j}{N}\right)^{\frac{2\delta}{k}-1}\cdot C_{2}\norm f_{\infty,1} \cdot \\
 & \left(R\cdot\left(\frac{Tj}{N}\right)^{\frac{2}{k}\left(s_{1}-\delta\right)}+R\left(1+R\right)^{\frac{3}{4}}\left(\frac{Tj}{N}\right)^{\frac{1}{4k}\left(1-2\delta\right)}\log\left(T\right)^{\frac{1}{4}}\right)\nonumber \\
 & +\frac{2\delta}{k}\sum_{j=1}^{N-1}\frac{1}{N}\left(\frac{j}{N}\right)^{\frac{2\delta}{k}-1}C_{2}\norm f_{\mathrm{Lip}\alpha}\left(R^{-\alpha}\log\left(R\right)+\left(\frac{Tj}{N}\right)^{-\alpha}\right).\label{eq:diff-N-T-R}
\end{align}
We have,
\[
\sum_{j=1}^{N-1}\frac{1}{N}\left(\frac{j}{N}\right)^{\frac{2\delta}{k}-1}\left(\frac{j}{N}\right)^{\frac{2}{k}\left(s_{1}-\delta\right)}=\sum_{j=0}^{N-1}\frac{1}{N}\left(\frac{j}{N}\right)^{\frac{2s_{1}}{k}-1}.
\]
The last expression is bounded by the integral $\int_{0}^{1}t^{\frac{2s_{1}}{k}-1}\,\d t$,
which is finite. Similarly,
\[
\sum_{j=1}^{N-1}\frac{1}{N}\left(\frac{j}{N}\right)^{\frac{2\delta}{k}-1}\left(\frac{j}{N}\right)^{\frac{1}{4k}\left(1-2\delta\right)}=\sum_{j=0}^{N-1}\frac{1}{N}\left(\frac{j}{N}\right)^{\frac{1+6\delta}{4k}-1}
\]
which is bounded by the integral $\int_{0}^{1}t^{\frac{1+6\delta}{4k}-1}\,\d t$,
which is also finite. Then, there is some constant $C_{3}$ which
depends only on $k$ and $\Gamma$ such that the second row of \ref{eq:diff-N-T-R}
is bounded by
\[
\frac{2\delta}{k}C_{3}C_{2}\norm f_{\infty,1}\left(R^{2}T^{\frac{2}{k}\left(s_{1}-\delta\right)}+R^{2}\left(1+2R\right)^{\frac{3}{4}}T^{\frac{1}{4k}\left(1-2\delta\right)}\log\left(T\right)^{\frac{1}{4}}\right).
\]
Since $0<\frac{1}{N}\le\frac{j}{N}$, the third row of \ref{eq:diff-N-T-R}
is bounded by
\[
\frac{2\delta}{k}\sum_{j=1}^{N-1}\frac{1}{N}\left(\frac{j}{N}\right)^{\frac{2\delta}{k}-1}C_{2}\norm f_{\mathrm{Lip}\alpha}\left(R^{-\alpha}\log\left(R\right)+N^{\alpha}T^{-\alpha}\right).
\]
If we take $C_{3}$ to be also larger than the integral $\int_{0}^{1}t^{\frac{2\delta}{k}-1}$,
we get that this expression is bounded by
\[
\frac{2\delta}{k}C_{3}\norm f_{\mathrm{Lip}\alpha}\left(R^{-\alpha}\log\left(R\right)+N^{\alpha}T^{-\alpha}\right).
\]
In total, for $T>NT_{0}$,
\begin{align*}
\left|\nu_{T}\left(f\right)-\int f\,\d\nu_{\Gamma}\right| & \le C_{1}\norm f_{\infty,1}\left(N^{-\frac{2\delta}{k}}+N^{-1}\log N\right)+\left(C_{\Gamma}+C\right)\norm f_{\mathrm{Lip}\alpha}N^{-\alpha}\\
 & +\frac{2\delta}{k}C_{3}C_{2}\norm f_{\infty,1}\left(R^{2}T^{\frac{2}{k}\left(s_{1}-\delta\right)}+R^{2}\left(1+2R\right)^{\frac{3}{4}}T^{\frac{1}{4k}\left(1-2\delta\right)}\log\left(T\right)^{\frac{1}{4}}\right)\\
 & +\frac{2\delta}{k}C_{3}\norm f_{\mathrm{Lip}\alpha}\left(R^{-\alpha}\log\left(R\right)+N^{\alpha}T^{-\alpha}\right).
\end{align*}
Write this as
\begin{align*}
\left|\nu_{T}\left(f\right)-\int f\,\d\nu_{\Gamma}\right| & \le C_{1}\norm f_{\infty,1}\left(N^{-\frac{2\delta}{k}}+N^{-1}\log N\right)+\left(C_{\Gamma}+C\right)\norm f_{\mathrm{Lip}\alpha}N^{-\alpha}\\
 & +\frac{2\delta}{k}C_{3}\norm f_{\mathrm{Lip}\alpha}N^{\alpha}T^{-\alpha}\\
 & +\frac{2\delta}{k}C_{3}C_{2}\norm f_{\infty,1}\left(R^{2}T^{\frac{2}{k}\left(s_{1}-\delta\right)}+R^{2}\left(1+2R\right)^{\frac{3}{4}}T^{\frac{1}{4k}\left(1-2\delta\right)}\log\left(T\right)^{\frac{1}{4}}\right)\\
 & +\frac{2\delta}{k}C_{3}\norm f_{\mathrm{Lip}\alpha}R^{-\alpha}\log\left(R\right).
\end{align*}
Taking the optimal choice
\begin{align*}
N & =T^{\frac{1}{2}}
\end{align*}
The condition $T>NT_{0}$ becomes $T>T_{0}^{2}$. Then, there is some
constant $C_{4}$ depending only on $\Gamma$ and $k$ such that,
for $T>T_{0}^{2}$ and for $R>0$, we have:
\begin{align*}
\left|\nu_{T}\left(f\right)-\int f\,\d\nu_{\Gamma}\right| & \le C_{4}\left(\norm f_{\infty,1}\left(T^{-\frac{\delta}{k}}+T^{-\frac{1}{2}}\log T\right)+\norm f_{\mathrm{Lip}\alpha}T^{-\frac{\alpha}{2}}\right)\\
 & +\frac{2\delta}{k}C_{3}C_{2}\norm f_{\infty,1}\left(R^{2}T^{\frac{2}{k}\left(s_{1}-\delta\right)}+R^{2}\left(1+2R\right)^{\frac{3}{4}}T^{\frac{1}{4k}\left(1-2\delta\right)}\log\left(T\right)^{\frac{1}{4}}\right)\\
 & +\frac{2\delta}{k}C_{3}\norm f_{\mathrm{Lip}\alpha}R^{-\alpha}\log\left(R\right).
\end{align*}
If $s_{1}-\delta>\frac{1}{8}\left(1-2\delta\right)$, we take $R=T^{-\frac{1}{2k}\left(s_{1}-\delta\right)}$
and in this case the second and third rows are bounded by a constant
times
\[
\left(\norm f_{\infty,1}+\norm f_{\mathrm{Lip}\alpha}\right)\left(T^{\frac{1}{k}\left(s_{1}-\delta\right)}+T^{\frac{\alpha}{2k}\left(s_{1}-\delta\right)}\right)\log\left(T\right).
\]
Since $\alpha\in\left(0,1\right]$, we always have $\frac{\alpha}{2k}\left(s_{1}-\delta\right)>\frac{1}{k}\left(s_{1}-\delta\right)$
(recall that $s_{1}-\delta<0$). Therefore, the total bound we get
in this case is
\[
\left(\norm f_{\infty,1}+\norm f_{\mathrm{Lip}\alpha}\right)T^{\frac{\alpha}{2k}\left(s_{1}-\delta\right)}\log\left(T\right).
\]
If $s-\delta\le\frac{1}{8}\left(1-2\delta\right)$, we take $R=T^{-\frac{1}{16k}\left(1-2\delta\right)}$,
and then the second and third rows are bounded by some constant times
\[
\left(\norm f_{\infty,1}+\norm f_{\mathrm{Lip}\alpha}\right)\left(T^{\frac{1}{8k}\left(1-2\delta\right)}\log\left(T\right)^{\frac{1}{4}}+T^{\frac{\alpha}{16k}\left(1-2\delta\right)}\log\left(T\right)\right).
\]
Since $\alpha\in\left(0,1\right]$, the total bound in this case is
\[
\left(\norm f_{\infty,1}+\norm f_{\mathrm{Lip}\alpha}\right)T^{\frac{\alpha}{16k}\left(1-2\delta\right)}\log\left(T\right)
\]
Combining all these estimates, we get that there is a constant $C_{5}>0$,
depending only on $\Gamma$ and $\rho$, such that for $T>T_{0}^{2}$,
\begin{align*}
\left|\nu_{T}\left(f\right)-\int f\,\d\nu_{\Gamma}\right| & \le C_{5}\left(\norm f_{\infty,1}\left(T^{-\frac{\delta}{k}}+T^{-\frac{1}{2}}\log T\right)+\norm f_{\mathrm{Lip}\alpha}T^{-\frac{\alpha}{2}}\right)\\
 & +C_{5}\left(\norm f_{\infty,1}+\norm f_{\mathrm{Lip}\alpha}\right)\left(T^{\frac{\alpha}{2k}\left(s_{1}-\delta\right)}+T^{\frac{\alpha}{16k}\left(1-2\delta\right)}\right)\log\left(T\right).
\end{align*}
Since $\frac{1}{2}<\delta\le1$ and $0<\alpha\le1$, we have
\[
\left|\frac{\alpha}{16k}\left(1-2\delta\right)\right|\le\frac{\alpha}{16}.
\]
Which is smaller than the exponents $\frac{\alpha}{2}$, and $\frac{1}{2}$.
Also, $\delta>\delta-s_{1}>0$ and therefore
\[
\left|\frac{\alpha}{2k}\left(s_{1}-\delta\right)\right|\le\frac{\delta}{k}.
\]
This means that the error $\left(T^{\frac{\alpha}{2k}\left(s_{1}-\delta\right)}+T^{\frac{\alpha}{16k}\left(1-2\delta\right)}\right)\log\left(T\right)$
is always larger than $T^{-\frac{\delta}{k}}+T^{-\frac{1}{2}}\log T+T^{-\frac{\alpha}{2}}$.
We deduce that as $T$ tends to infinity,
\[
\nu_{T}\left(f\right)=\int f\,\d\nu_{\Gamma}+\o\left(\left(\norm f_{\infty,1}+\norm f_{\mathrm{Lip}\alpha}\right)\left(T^{\frac{\alpha}{2k}\left(s_{1}-\delta\right)}+T^{\frac{\alpha}{16k}\left(1-2\delta\right)}\right)\log\left(T\right)\right),
\]
where the implied constant depends only on $\Gamma$ and $\rho$.
This concludes the proof of theorem \ref{thm:Main_Quant}.

\section*{Appendix - Convergence rate of Fej\'{e}r means\label{sec:Appendix---Convergence}}

In this appendix we derive a quantitative estimate for the approximation
of a H\"{o}lder continuous function on $\mathbb{T}^{2}$ by its Fej\'{e}r
means. This result is certainly well known and follows easily from classical results as in \cite{zygmund1968trigonometric}. The proof is included here for convenience of the reader.

Let $\psi:\left[0,2\pi\right]\times\left[0,2\pi\right]\rightarrow\c$
be H\"{o}lder continuous with exponent $\alpha\in\left(0,1\right]$ and
constant $C_{\psi}$. We wish to approximate $\psi$ by trigonometric
polynomials in a quantitative fashion. We will approximate $\psi$
by its Fej\'{e}r means, which are Cesàro means of Fourier series.

For $N,M\in\nat$, define the $\left(N,M\right)$-th Fourier partial
sum by
\[
s_{N,M}\left(\theta_{1},\theta_{2}\right)=\sum_{n=-N}^{N}\sum_{m=-M}^{M}\hat{\psi}\left(n,m\right)e^{in\theta_{1}}e^{im\theta_{2}},
\]
where 
\[
\hat{\psi}\left(n,m\right)=\frac{1}{4\pi^{2}}\int_{\mathbb{T}^{2}}\psi\left(\theta_{1},\theta_{2}\right)e^{-\left(in\theta_{1}+im\theta_{2}\right)}\,\d\theta_{1}\d\theta_{2}.
\]
The $N$-th Dirichlet kernel $D_{N}:\left[0,2\pi\right]\rightarrow\real$
is defined by
\[
D_{N}\left(\theta\right)=\sum_{k=-N}^{N}e^{ik\theta}=\frac{\sin\left(\left(N-2\right)\theta\right)}{\sin\left(\frac{\theta}{2}\right)}.
\]
It is easy to verify that $s_{N,M}\psi$ is given by the convolution
of $\psi$ and the product of Dirichlet kernels as,
\begin{equation}
s_{N,M}\psi\left(\theta_{1},\theta_{2}\right)=\frac{1}{4\pi^{2}}\int_{\mathbb{T}^{2}}\psi\left(\theta_{1},\theta_{2}\right)D_{N}\left(\theta_{1}-u\right)D_{M}\left(\theta_{2}-v\right)\,\d u\d v.\label{eq:dirichlet-convolution}
\end{equation}
Define:
\[
\sigma_{R}\left(\theta\right)=\frac{1}{R^{2}}\sum_{N,M=0}^{R-1}s_{NM}\left(\theta\right).
\]
$\sigma_{R}$ is the Cesàro sum of the Fourier developments of $\psi$
with $N,M\le R$. It is called the $R$-th\emph{ Fej\'{e}r mean} of $\psi$.
We can also write this as
\[
\sigma_{R}\left(\theta_{1},\theta_{2}\right)=\sum_{n,m=-R}^{R}A_{n,m}\hat{\psi}\left(\theta_{1},\theta_{2}\right)e^{i\left(n\theta_{1}+m\theta_{2}\right)}
\]
for some constants $A_{n,m}$ such that $0\le A_{n,m}\le1$ for all
$m,n\in\z$. Then, using \ref{eq:dirichlet-convolution},
\begin{align*}
\sigma_{R}\left(\theta_{1},\theta_{2}\right) & =\frac{1}{R^{2}}\sum_{N,M=0}^{R-1}\frac{1}{4\pi^{2}}\int_{\mathbb{T}^{2}}\psi\left(u,v\right)D_{N}\left(\theta_{1}-u\right)D_{M}\left(\theta_{2}-v\right)\,\d u\d v\\
 & =\frac{1}{4\pi^{2}}\int_{\mathbb{T}^{2}}\psi\left(\theta_{1}-u,\theta_{2}-v\right)\frac{1}{R^{2}}\sum_{n,m=0}^{R-1}D_{N}\left(u\right)D_{M}\left(v\right)\,\d u\d v.
\end{align*}
Thus, $\sigma_{R}$ is obtained by a convolution of $\psi$ with the
kernel:
\[
\frac{1}{R^{2}}\sum_{n,m=0}^{R-1}D_{N}\left(u\right)D_{M}\left(v\right)=\left(\frac{1}{R}\sum_{N=0}^{R-1}D_{N}\left(u\right)\right)\left(\frac{1}{R}\sum_{M=0}^{R-1}D_{M}\left(v\right)\right).
\]
The function
\[
F_{R}\left(u\right)=\frac{1}{R}\sum_{N=0}^{R-1}D_{N}\left(u\right),
\]
 is called the Fej\'{e}r kernel. It is known (see \cite{zygmund1968trigonometric})
that
\[
F_{R}\left(u\right)=\frac{1}{R}\left(\frac{\sin\left(\frac{Ru}{2}\right)}{\sin\left(\frac{u}{2}\right)}\right)^{2}.
\]
Since $F_{R}\left(u\right)\ge0$ it follows that $\int_{\mathbb{T}^{2}}\left|F_{R}\right|=1$
for all $R$. We can now give an estimate for
\[
\norm{\sigma_{R}-\psi}=\sup_{\theta_{1},\theta_{2}\in\left[0,2\pi\right]}\left|\sigma_{R}\left(\theta_{1},\theta_{2}\right)-\psi\left(\theta_{1},\theta_{2}\right)\right|.
\]
For all $\theta_{1},\theta_{2}\in\left[0,2\pi\right]$,
\begin{align*}
\left|\sigma_{R}\left(\theta_{1},\theta_{2}\right)-\psi\left(\theta_{1},\theta_{2}\right)\right| & \le\frac{1}{4\pi^{2}}\int_{\mathbb{T}^{2}}\left|\psi\left(\theta_{1}-u,\theta_{2}-v\right)-\psi\left(\theta_{1},\theta_{2}\right)\right|F_{R}\left(u\right)F_{R}\left(v\right)\,\d u\d v\\
 & \le\frac{1}{4\pi^{2}}\int_{\mathbb{T}^{2}}C_{\psi}\left(\max\left\{ \left|u\right|,\left|v\right|\right\} \right)^{\alpha}F_{R}\left(u\right)F_{R}\left(v\right)\,\d u\d v.
\end{align*}
Using $\left(\max\left\{ \left|u\right|,\left|v\right|\right\} \right)^{\alpha}\le\left|u\right|^{\alpha}+\left|v\right|^{\alpha}$
and the symmetry between $v$ and $u$,
\begin{align*}
\left|\sigma_{R}\left(\theta_{1},\theta_{2}\right)-\psi\left(\theta_{1},\theta_{2}\right)\right| & \le\frac{C_{\psi}}{4\pi^{2}}\int_{\mathbb{T}^{2}}\left|u\right|^{\alpha}F_{R}\left(u\right)F_{R}\left(v\right)\,\d u\d v\\
 & +\frac{C_{\psi}}{4\pi^{2}}\int_{\mathbb{T}^{2}}\left|v\right|^{\alpha}F_{R}\left(u\right)F_{R}\left(v\right)\,\d u\d v\\
 & =\frac{C_{\psi}}{2\pi^{2}}\int_{\mathbb{T}^{2}}\left|v\right|^{\alpha}F_{R}\left(u\right)F_{R}\left(v\right)\,\d u\d v.
\end{align*}
Since $\int F_{R}\left(v\right)\,\d v=1$, we get
\[
\left|\sigma_{R}\left(\theta_{1},\theta_{2}\right)-\psi\left(\theta_{1},\theta_{2}\right)\right|\le\frac{C_{\psi}}{2\pi^{2}}\int_{0}^{2\pi}\left|u\right|^{\alpha}F_{R}\left(u\right)\,\d u.
\]
Since $\left|u\right|^{\alpha}F_{R}\left(u\right)$ is symmetric under
$\theta\mapsto\theta+\pi$, we have
\[
\left|\sigma_{R}\left(\theta_{1},\theta_{2}\right)-\psi\left(\theta_{1},\theta_{2}\right)\right|\le\frac{C_{\psi}}{\pi^{2}}\int_{0}^{\pi}u^{\alpha}F_{R}\left(u\right)\,\d u
\]
If we split the integral over $\left[0,\pi\right]$ to the intervals
$\left[0,\frac{1}{R}\right]$ and $\left[\frac{1}{R},\pi\right]$,
we get
\begin{align*}
\int_{0}^{\pi}u^{\alpha}F_{R}\left(u\right)\,\d u & =\int_{0}^{\frac{1}{R}}u^{\alpha}F_{R}\left(u\right)\,\d u+\int_{\frac{1}{R}}^{\pi}u^{\alpha}F_{R}\left(u\right)\,\d u.
\end{align*}
Now, for all $0\le x\le\frac{\pi}{2}$, $\left|\sin\left(x\right)\right|\ge\frac{\left|x\right|}{\pi}$.
So, for $0\le\theta\le\pi$, we have that $\left|\sin\left(\frac{\theta}{2}\right)\right|\ge\frac{\left|\theta\right|}{2\pi}.$
Then, since $\left|\sin\left(\frac{Ru}{2}\right)\right|\le1$ for
all $u$,
\begin{align*}
\int_{0}^{\pi}u^{\alpha}F_{R}\left(u\right)\,\d u & \le\int_{0}^{\frac{1}{R}}R^{-\alpha}F_{R}\left(u\right)\,\d u+\int_{\frac{1}{R}}^{\pi}u^{\alpha}\frac{1}{R}\left(\frac{\sin\left(\frac{Ru}{2}\right)}{\sin\left(\frac{u}{2}\right)}\right)^{2}\,\d u\\
 & \le R^{-\alpha}\int_{0}^{1}F_{R}\left(u\right)\,\d u+\frac{1}{R}\int_{\frac{1}{R}}^{\pi}u^{\alpha}\frac{4\pi^{2}}{u^{2}}\,\d u\\
 & =R^{-\alpha}+\frac{4\pi^{2}}{R}\int_{\frac{1}{R}}^{\pi}u^{\alpha-2}\,\d u.
\end{align*}
For $\alpha\in\left(0,1\right)$, 
\[
\int_{\frac{1}{R}}^{\pi}u^{\alpha-2}\,\d u=\frac{1}{\left(\alpha-1\right)}\left(1-R^{\alpha-1}\right).
\]
and for $\alpha=1$,
\[
\int_{\frac{1}{R}}^{\pi}u^{-1}\,\d u=\log\left(\pi R\right).
\]
In total,
\[
\int_{0}^{\pi}u^{\alpha}F_{R}\left(u\right)\,\d u=\begin{cases}
\o\left(R^{-\alpha}\right) & \alpha\in\left(0,1\right)\\
\o\left(\frac{\log R}{R}\right) & \alpha=1
\end{cases}.
\]
We have proved:
\begin{prop}
\label{prop:fejer-holder-error}If $\psi:\mathbb{T}^{2}\rightarrow\c$
is H\"{o}lder continuous with exponent $\alpha$ and constant $C_{\psi}$,
then for $R\in\nat$ and $\sigma_{R}$ the $R$-th Fej\'{e}r mean of $\psi$,
there is some constant $C_{0}$ such that
\[
\norm{\sigma_{R}-\psi}_{\infty}\le C_{\psi}C_{0}R^{-\alpha}\log R.
\]
\end{prop}
\printbibliography

\end{document}